\documentclass[11pt]{article}

\usepackage[margin=1in]{geometry}

\usepackage{amsmath}
\usepackage{amssymb}
\usepackage{amsthm}
\usepackage{bm}
\usepackage{color}
\usepackage{graphicx}
\usepackage{mathtools}
\usepackage{hyperref}
\usepackage{enumerate}
\usepackage[shortlabels]{enumitem}
\usepackage[colorinlistoftodos,prependcaption,textsize=footnotesize]{todonotes}

\usepackage{hyperref}
\hypersetup{
    colorlinks=true,     
    linkcolor=blue,      
    citecolor=blue,      
    filecolor=blue,      
    urlcolor=blue        
}

\newtheorem{definition}{Definition}[section]
\newtheorem{lemma}[definition]{Lemma}
\newtheorem{proposition}[definition]{Proposition}

\newtheorem{example}[definition]{Example}

\newtheorem{corollary}[definition]{Corollary}
\newtheorem{conjecture}[definition]{Conjecture}
\newtheorem{theorem}[definition]{Theorem}
\newtheorem*{proposition*}{Proposition}
\theoremstyle{remark}
\newtheorem{remark}[definition]{Remark}

\newcommand{\spanVec}{\ensuremath{\mathrm{span}\,}}

\newcommand{\reInt}{\mathrm{ri}\,}
\newcommand{\SOC}[2]{{\mathcal{L}^{#2} _{#1}}}
\newcommand{\norm}[1]{\lVert{#1}\rVert}
\newcommand{\T}{T} 
\newcommand{\inProd}[2]{\langle #1 , #2 \rangle }

\newcommand{\cS}{\mathcal{S}}

\newcommand{\cppcone}[1]{\mathrm{CP}^{#1}}
\newcommand{\cscone}[1]{\mathrm{CS}^{#1}}
\newcommand{\psdcone}[1]{\cS_+^{#1}}
\newcommand{\dnncone}[1]{\mathrm{DNN}^{#1}}
\renewcommand{\S}{\cS}

\newcommand{\RR}{\mathbb{R}}

\newcommand{\rank}{\textup{rank}}
\newcommand{\tr}{\textup{tr}}

\newcommand{\stdCone}{ {\mathcal{K}}}
\newcommand{\stdFace}{F}
\newcommand{\ambSpace}{ \mathcal{V} }
\newcommand{\zeroes}{\operatorname{zeroes}}
\newcommand{\coner}{\operatorname{cone\,rows}}

\newcommand{\conec}{\operatorname{cone\,cols}}
\newcommand{\euc}{\textup{Euc}}
\renewcommand{\Re}{\mathbb{R}}   

\title{Self-dual polyhedral cones and their slack matrices}
\numberwithin{equation}{section}
\author{
    Jo\~ao Gouveia\thanks{CMUC, Department of Mathematics, University of Coimbra, 3001-454 Coimbra, Portugal.
    This author was supported by the Centre for Mathematics of the University of Coimbra (UIDB/00324/2020,
    funded by the Portuguese Government through FCT/MCTES)
    Email: \href{jgouveia@mat.uc.pt }{jgouveia@mat.uc.pt}.} \and
	Bruno F. Louren\c{c}o\thanks{Department of Statistical Inference and Mathematics, Institute of Statistical Mathematics, Japan.
	This author was supported partly by the JSPS Grant-in-Aid for Early-Career Scientists  JP19K20217 and the Grant-in-Aid for Scientific Research (B) JP21H03398.
	Email: \href{bruno@ism.ac.jp}{bruno@ism.ac.jp}.}
}

\begin{document}	
\maketitle
\begin{abstract}
We analyze self-dual polyhedral cones and prove several properties about their slack matrices.
In particular, we show that self-duality is equivalent to the existence of a positive semidefinite (PSD) slack.
Beyond that, we show that if the underlying cone is irreducible, then the corresponding PSD slacks are not only doubly nonnegative matrices (DNN) but are extreme rays of the cone of DNN matrices, which correspond to a family of extreme rays not previously described. 
More surprisingly, we show that, unless the cone is simplicial, PSD slacks not only fail to be completely positive matrices but they also lie outside the cone of completely positive semidefinite matrices. Finally, we show how one can use semidefinite programming to probe the existence of self-dual cones with given combinatorics. Our results are given for polyhedral cones but we also discuss some consequences for negatively self-polar polytopes.
\end{abstract}

\section{Introduction}

Slack matrices of polytopes or polyhedral cones codify their structure in matrix form \cite{slack}. Since their introduction in \cite{Ya91} they have provided a fundamental tool for the study of the complexity of combinatorial optimization problems. 
For example, obtaining bounds to certain factorization ranks (such as the nonnegative rank or the semidefinite rank) for these matrices provides a way to bound the extension complexity of the underlying polytopes \cite{FMPTHW12,GPT13,LRS15}.
Slack matrices have also been used to attack problems of realizability or projective uniqueness for polytopes \cite{GMRW20,GMA23}.

In this work, we use the study of self-duality of convex cones, a classical subject in both convex and discrete geometry \cite{BF76,Bar81,grunbaum1988selfduality,lovasz1983self}, to make a surprising connection between the theory of slack matrices and the geometry of some important matrix cones. It turns out that slack matrices of self-dual polyhedral cones are deeply connected with the extreme rays of the doubly nonnegative matrices, the completely positive matrices and, more generally, the completely positive semidefinite matrices. These three matrix cones are widely studied examples of convex cones with non-trivial  geometry  \cite{BP94,BS03,laurent2015conic}. They have important theoretical applications and have been a fruitful target for research for several decades now. 

We focus our attention on polyhedral cones  that are self-dual with respect to some inner product. The most classical example is the nonnegative orthant $\Re^n_+$, but there are many more such cones. For example, it is possible to construct polyhedral self-dual cones in $\Re^3$ generated by regular polygons with any odd number of vertices, see \cite[pg.~152]{BF76} for the details.
 Our main results are as follows.
\begin{enumerate}[$(i)$]
	\item A polyhedral cone is self-dual with respect to some inner product if and only if one of its slack matrices is positive semidefinite (PSD). Furthermore, if one of its slack matrices is PSD then, rearranging the rows and rescaling the columns if necessary, \emph{all} of them are PSD, see Theorem~\ref{thm:self_psd}. 
		This fits in with the general philosophy that geometric properties of  polyhedral sets get translated to linear algebraic properties of slack matrices. 
	\item Because slack matrices have nonnegative entries, the PSD slacks of self-dual cones are actually doubly nonnegative (DNN) matrices. It turns out that when the underlying cone is \emph{irreducible} the PSD slack matrices are not only DNN matrices, but they must be extreme rays of the cone of DNN matrices, see Theorem~\ref{theo:slack_ext}.
	For $5\times 5$ DNN matrices this leads to an interesting characterization of extreme rays in terms of self-dual cones over a pentagon, see Theorem~\ref{theo:dnn5}. And, more generally, we can generate new families of extreme rays of the DNN matrices distinct from the ones constructed in \cite{HL96}. 
	
	\item 
	We show that non-diagonal PSD slack matrices of self-dual polyhedral cones not only lie outside the cone of completely positive (CP) matrices, but they also do not belong to the cone of \emph{completely positive semidefinite} matrices, see Theorem~\ref{theo:diag_slack} and Corollary~\ref{col:cps}.
	The cone of completely positive semidefinite matrices  is a far reaching generalization of CP matrices with many interesting applications \cite{laurent2015conic}.
	Through our results, we can generate many examples of DNN matrices that fail to be completely positive semidefinite.
	
	\item Finally, we use the new characterization of self-duality to propose a semidefinite programming approach to numerically search for self-dual realizations of a given combinatorial type of polyhedral cone. We then apply this framework to enumerate low dimensional self-dual polyhedral cones with few extreme rays.
	
\end{enumerate} 
This paper is organized as follows.
In Section~\ref{sec:prel} we discuss the notation and some background material. In Section~\ref{sec:slack}, we present the basic results concerning the slack matrices of self-dual polyhedral cones and provide an alternative interpretation of these results in terms of polytopes. Then, 
in Section~\ref{sec:dnn_cs} we prove our results concerning the extreme rays of the doubly nonnegative cone and the completely positive semidefinite matrices. 
Finally, some numerical explorations obtained by means of semidefinite programming and future perspectives are presented in Section~\ref{sec:exp}.

\section{Preliminaries}\label{sec:prel}
Let $\stdCone \subseteq \Re^d$ be a closed convex cone and let $\inProd{\cdot}{\cdot}$ be an inner product for $\Re^d$. Then, 
the dual cone of $\stdCone$ with 
respect to $\inProd{\cdot}{\cdot}$ is
\begin{equation}\label{eq:dual_cone}
\stdCone^* \coloneqq \{y \in \Re^d \mid \inProd{x}{y} \geq 0, \forall x \in \stdCone \}.
\end{equation}
We recall that the \emph{lineality space} of $\stdCone$ is $\stdCone \cap - \stdCone$ and is the largest subspace contained in $\stdCone$.
We also have the following relation
\begin{equation}\label{eq:lineality}
\stdCone \cap - \stdCone = (\stdCone^{*})^\perp,
\end{equation}
where the notation $C^\perp$ indicates elements orthogonal to a subset $C \subseteq \Re^d$ with respect the inner product $\inProd{\cdot}{\cdot}$.
The span and dimension of $\stdCone$ will be denoted by $\spanVec \stdCone$ and $\dim \stdCone$, respectively.
With that, $\stdCone$ is said to be \emph{pointed} if $\stdCone \cap - \stdCone = \{0\}$ and \emph{full-dimensional} if $\dim \stdCone = d$.
If $\stdCone$ can be written as the set of solutions of finitely many linear equalities and inequalities then $\stdCone$ is said to be \emph{polyhedral}.

A \emph{face} of $\stdCone$ is a convex cone $\stdFace \subseteq \stdCone$
such that the condition ``$x,y \in \stdCone, x+y \in \stdFace$'' implies $x,y \in \stdFace$. An \emph{extreme ray} is a face of the form $\{\alpha x \mid \alpha \geq 0\}$ for some $x \in \stdCone$. In this case, we say that $\stdFace$ is the extreme ray generated by $x$. A \emph{facet} is a face of $\stdCone$ that has dimension $\dim \stdCone -1$.

We denote the set of real $n\times m$ matrices by $\Re^{n\times m}$. If $A \in \Re^{n\times m}$, the transpose and inverse (if exists) of $A$ will be denoted by $A^\T$ and $A^{-1}$, respectively.
We denote by \emph{$\coner(A)$} the convex cone in $\Re^m$ generated by its rows and by \emph{$\conec(A)$} the convex cone in $\Re^n$ generated by its columns.

Given closed convex cones $\stdCone \subseteq \Re^{d_1}$ and $\stdCone' \subseteq \Re^{d_2}$ we say that \emph{$\stdCone$ and $\stdCone'$ are linearly isomorphic} if there exists a linear bijection $A: \spanVec \stdCone \mapsto \spanVec\stdCone'$ such that 
$A(\stdCone) = \stdCone'$. 
If $d_1 = d_2$, the dual cones are related as follows 
\begin{equation}\label{eq:dual_iso}
\stdCone'^* = (A^*)^{-1} \stdCone^*,
\end{equation}
where $A^*$ the adjoint of $A$ according to the underlying inner product.

The space of $n\times n$ real symmetric matrices will be denoted by $\S^n$. 
The cone of positive semidefinite (PSD) matrices in $\S^n$ will be denoted by $\psdcone{n}$.
The doubly nonnegative matrices (i.e., matrices that are PSD and nonnegative) in $\S^n$ will be denoted by $\dnncone{n}$.
Finally, we denote by $\cppcone{n}$ the  cone of $n\times n$ completely positive matrices, which we recall that consists of the $n\times n$ matrices $A$ admitting a decomposition $A = XX^\T$  for some $X \in \Re^{n\times m}$ such all the entries of $X$ are nonnegative.

\subsection{Inner products and self-duality}\label{sec:self_dual}

We will use the notation $\stdCone^*_{\euc}$ to indicate the dual cone of $\stdCone$ obtained under the Euclidean inner product  so that 
\[
\stdCone^*_{\euc} \coloneqq \{y \in \Re^d \mid x^{\T}y \geq 0, \forall x \in \stdCone \}.
\]
We note that if the inner product $\inProd{\cdot}{\cdot}$ changes then the dual cone $\stdCone^*$ changes as well. In particular, if $A$ is the positive definite matrix associated to some inner product $\inProd{\cdot}{\cdot}$, the corresponding dual of $\stdCone$ can be written as 
\begin{equation}\label{eq:dual_cone2}
\stdCone^* \coloneqq \{y \in \Re^d \mid x^\T A y \geq 0, \forall x \in \stdCone \} = A^{-1} \stdCone^*_{\euc}.
\end{equation}
Conversely, if $\hat \stdCone = A^{-1} \stdCone^*_{\euc}$ for some positive definite matrix $A$, then the dual of $\stdCone$ under the inner product induced by $A$ is precisely $\hat \stdCone$.
In conclusion, all possible distinct dual cones of $\stdCone$ are given by 
$A^{-1}\stdCone^*_{\euc}$, where $A$ is a symmetric positive definite matrix.

With these subtleties in mind, we define self-duality as follows.

\begin{definition}[Self-dual cones]\label{def:self_dual}
 A closed convex cone $\stdCone \subseteq \Re^d$ is said to be \emph{self-dual} if there exists some inner product $\inProd{\cdot}{\cdot}$ under which 
 $\stdCone = \stdCone^*$. In this case, we say that $\stdCone$ is self-dual with respect to $\inProd{\cdot}{\cdot}$.
\end{definition}
Typical examples of self-dual cones include the nonnegative orthant and the positive semidefinite matrices. However, there are many more self-dual cones, as we will see throughout this paper, see also \cite{BF76}.

We remark that in the Definition~\ref{def:self_dual}, it is entirely possible that $\stdCone$ is self-dual but $\stdCone \neq \stdCone^*_{\euc}$. One such example are the so-called power cones, see \cite[Section~4.3.1]{Ch09}.
Still, self-duality can be characterized in terms of $\stdCone^*_{\euc}$ as follows.
\begin{proposition}[{e.g., \cite[Proposition~1]{IL19}}]\label{prop:self_dual}
	Let $\stdCone \subseteq \Re^d$ be a closed convex cone. Then, there exists an inner product on $\Re^d$ such that $\stdCone = \stdCone^*$ if and only if there exists a symmetric positive definite matrix $A$ such that 
	$A\stdCone = \stdCone^*_{\euc}$. 
\end{proposition}
Proposition~\ref{prop:self_dual} has the following consequence. In order to show that $\stdCone$ is \emph{not} self-dual, it does not suffice to merely check that $\stdCone \neq \stdCone^*_{\euc}$, one must go through the harder task of showing that $\stdCone$ is not linearly isomorphic
to $\stdCone^*_{\euc}$ through a positive definite matrix. It might be fair to say that this latter  task is significantly harder than the former. 

For example, for the $p$-cones $
\SOC{p}{n+1}\coloneqq\{(t,x)\in\Re\times \Re^n \mid t \geq \norm{x}_p\}$, where $\norm{\cdot}_p$ is the $p$-norm, it is straightforward to verify that
if $p \in (1,\infty)$, $p \neq 2$ and $n\geq 2$ then $\SOC{p}{n+1} \neq  (\SOC{p}{n+1})^*_{\euc} = \SOC{q}{n+1}$, where $1/p + 1/q = 1$. 
However, checking that $\SOC{p}{n+1}$ is \emph{never} self-dual with respect to any inner product requires more work and is a consequence of the fact that for those $p$ and $n$ the cone $\SOC{p}{n+1}$ is never isomorphic to its dual cone, see \cite[Theorem~11 and Corollary~14]{IL19} or \cite[Section~5.2]{LLP21}. 
A more subtle type of failure happens for the cone $\SOC{1}{3}$ which is \emph{actually} isomorphic to its Euclidean dual cone $\SOC{\infty}{3} \coloneqq \{(t,x) \in \Re\times \Re^2 \mid t \geq \norm{x}_{\infty}\}$ but the isomorphism cannot be realized through a positive definite matrix so $\SOC{\infty}{3}$ is not self-dual, e.g., see \cite[Equation~(2) and the proof of Corollary~14]{IL19}.
These subtleties are important because, for example, a cone may become \emph{symmetric}  under a change of inner product \cite{Or20}, and symmetric cones enjoy many favourable theoretical properties \cite{FK94,Fa08}.


\section{Self-dual polyhedral cones}\label{sec:slack}
In this work we will focus on the polyhedral cones that are self-dual according to Definition~\ref{def:self_dual}.
The main tool we will use for our explorations will be \emph{slack matrices}, so we first present a basic discussion on their properties.

\subsection{Slack matrices and basic properties}
We start with the following definition.
\begin{definition}[Slack matrices]\label{def:slack}
	Let $\stdCone \subseteq \Re^d$ be  a pointed full-dimensional polyhedral cone and let $\stdCone^*$ be the dual cone with respect some inner product.
	Let $n$ and $m$ be the number of extreme rays 
	of $\stdCone$ and $\stdCone^*$, respectively.
	A \emph{slack matrix} of $\stdCone$ is a  matrix  $M \in \Re^{n\times m}$ constructed as follows. Let $\{x_1,\ldots, x_n\}$ and $\{y_1,\ldots, y_m\}$ be 
	sets of generators for all the extreme rays of  $\stdCone$ and $\stdCone^*$ respectively. Then 
	$M_{ij} = \inProd{x_i}{y_j}$ for $i \in \{1,\ldots, n\}, j \in \{1,\ldots, m\}$.
	
	The set of slack matrices of $\stdCone$ is denoted by 
	$\S(\stdCone)$.
\end{definition}
\begin{remark}\label{rem:def_slack}
The notion of slack matrix is more commonly used for polytopes, and there is some variation on how it is defined, specially in the case of cones. Our definition of slack matrices is slightly more strict than, say, the one in \cite{slack}, where redundancies are allowed.
For example, in \cite{slack} it is possible that a slack matrix has a repeated row or has a zero row. 

Here, however, we insist in Definition~\ref{def:slack} that the $x_i$ and $y_j$ correspond to distinct extreme rays of $\stdCone$ and $\stdCone^*$.  This definition however is the natural generalization of the most common notion of slack matrix for polytopes and is more natural in the context of this paper. It coincides with the definitions used in \cite{wang2020lifts,gouveia2022combining}.
In particular, no two rows of a slack matrix can have the same pattern of zeros since it would imply that different extreme rays of $\stdCone$ are orthogonal to the exact same extreme rays of $\stdCone^*$.
In addition, since $\stdCone$ is assumed to be pointed, a row of zeros would imply that some extreme ray is orthogonal to $\stdCone^*$, but $\stdCone^{*\perp} = (-\stdCone) \cap \stdCone$, which is $\{0\}$ by assumption. 
Analogously, since $\stdCone$ is full-dimensional, no columns of zeros are possible, since it would imply that some extreme ray of $\stdCone^*$ is orthogonal to $\stdCone$.
In summary, the assumptions we make in Definition~\ref{def:slack} will be helpful to avoid trivial cases arising from redundancies in the slack matrix.

\end{remark}
In Definition~\ref{def:slack}, we did not specify the inner product used when constructing slack matrices. We will show in the next proposition that this is not an issue and that linearly isomorphic cones have the same set of slack matrices.

\begin{proposition}[Invariance of slack matrices]\label{prop:slack_inv}
Let $\stdCone \subseteq \Re^d$ be  a pointed full-dimensional polyhedral cone. 	The following items hold.
\begin{enumerate}[$(i)$]
	\item\label{prop:slack_inv:1} $\S(\stdCone)$ does not depend on the choice of inner product in Definition~\ref{def:slack}.
	\item\label{prop:slack_inv:2} If $\stdCone$ is linearly isomorphic to $\stdCone'$ then 
	$\S(\stdCone) = \S(\stdCone')$.
\end{enumerate}
\end{proposition}
\begin{proof}
\fbox{\ref{prop:slack_inv:1}}
A quick proof of this item can be obtained by observing that 
elements in $\stdCone^*$  correspond to the linear functionals (i.e., elements in the dual space of $\Re^d$) that are nonnegative over $\stdCone$. This shows that there is no dependency on the inner product in Definition~\ref{def:slack} and different inner products merely correspond to selecting different representatives to each linear functional.
 
For completeness, here is also a detailed proof. Let $M \in \S(\stdCone)$ as in Definition~\ref{def:slack} and let $A$ be the symmetric positive definite matrix of the corresponding inner product. In view of \eqref{eq:dual_cone2}, the generators of the extreme rays of $\stdCone^*$ can be chosen to be of the form $\{A^{-1}\hat y_1, \ldots A^{-1}\hat y_m\}$, where the $\hat y_i$ generate the extreme rays of $\stdCone_{\euc}^*$ (the dual cone with respect the Euclidean inner product) and $y_i = A^{-1}\hat y_i$.
Therefore, the entries of $M$ satisfy
\[
M_{ij} = \inProd{x_i}{y_j} = x_i^\T A (A^{-1}\hat y_j) =  x_i ^\T \hat y_j,
\]
which does not depend on the particular choice of the inner product.

\fbox{\ref{prop:slack_inv:2}} Suppose that $ \stdCone = A\stdCone'$ holds for some bijective linear map $A$. From \eqref{eq:dual_iso}
\[
\stdCone_{\euc}^* = A^{-\T}(\stdCone')_{\euc}^*
\]
holds.
Let $M \in \S(\stdCone)$. By item~\ref{prop:slack_inv:1}, we may consider the Euclidean inner product without loss of generality so that we have
\[
M_{ij} = x_i^\T y_j = (A\hat x_i)^\T (A^{-\T}\hat y_j) = \hat x_i ^\T \hat y_j,
\]
where the $\hat x_i$ and $\hat y_j$ generate the extreme rays of $\stdCone'$ and $(\stdCone')_{\euc}^*$.
So, the entries of $M_{ij}$ do not depend on the particular isomorphism considered.
\end{proof}
Next, we need the following technical lemma.

\begin{lemma}\label{lem:iso}
Suppose that $A,U \in \Re^{n\times d}$ and 
$B, V \in \Re^{d\times m}$ are  such that 
$A,B,U,V$ have rank $d$	and 
$
AB = UV
$
holds. Then, there exists a nonsingular  matrix $S \in \Re^{d\times d}$ such that $AS = U$ and $SV = B$.
\end{lemma}
\begin{proof}
We start by noting that under the conditions of the lemma, $AB$ has rank $d$.
To see this pick non-singular $d \times d$ submatrices $A_I$ and $B_J$ of $A$ and $B$, obtained by taking the $d$ rows of $A$ indexed by $I$ and the $d$ columns of $B$ indexed by $J$ respectively. Then $A_I B_J$ is non-singular and is the submatrix of $AB$ with columns indexed by $I$ and rows indexed by $J$.

Denote the column space of $A$ by $\textup{col}(A)$.
Since the columns of $AB$ are linear combinations of the columns of $A$, and since their column spaces have both dimension $d$, we must have $\textup{col}(A)=\textup{col}(AB)$, and since $AB=UV$ we also must have $\textup{col}(AB)=\textup{col}(U)$. Therefore the columns of $A$ and the columns of $U$ are both bases to the same vector space, which implies that there is a non-singular change of basis matrix $S$ such that $AS=U$.

Therefore, $ASV = UV = AB$, which leads to 
$A(SV - B) = 0$. Since $A$ has rank $d$, this implies that $SV = B$.
\end{proof}
We now move on to our main result on slack matrices, which is a slightly refined version of \cite[Lemma~5]{slack} more suitable for our purposes. The idea is that a factorization of a slack matrix $M$ of the form $M = UV$ leads to a  linear isomorphism between $\stdCone$ and the cone generated by the rows of $U$. Furthermore, the factorization also identifies the corresponding dual cone.

\begin{theorem}\label{theo:slack}
	Let $\stdCone \subseteq \Re^d$ be  a pointed full-dimensional polyhedral cone and let $M$ be a slack matrix of $\stdCone$.
	Suppose that $M = UV$, where $U \in \Re^{n\times d}$, $V \in \Re^{d\times m}$. Then, the following items holds.
	\begin{enumerate}[$(i)$]
		\item\label{theo:slack:i} $\stdCone$ is linearly isomorphic to $\coner(U)$.
		\item\label{theo:slack:ii} $(\coner(U))^*_{\euc} = \conec(V)$.
	\end{enumerate}
\end{theorem}
\begin{proof}
For simplicity, let $\stdCone' \coloneqq \coner(U)$ and $\hat \stdCone \coloneqq \conec(V)$.

Since $M$ is a slack matrix of $\stdCone$, there exists $A \in \Re^{n\times d}, B \in \Re^{d\times m}$, such that $AB = M$ and 
the rows of $A$ correspond to the extreme rays of $\stdCone$  and the columns of $B$ correspond to the extreme rays of $\stdCone^*_{\euc}$.

Since $AB = UV$, by Lemma~\ref{lem:iso} there exists a nonsingular $S \in \Re^{d\times d}$ such that  $AS = U$ and $SV = B$. This implies that $S^\T \stdCone  = \stdCone'$ which proves item~\ref{theo:slack:i} and, using \eqref{eq:dual_iso}, also  leads to 
\begin{equation}\label{eq:theo1a}
\stdCone'^*_{\euc} = S^{-1}\stdCone^*_{\euc}.
\end{equation}
Since $SV = B$, we have that $S$ maps the cone generated by the columns of $V$ (i.e., $\hat \stdCone$) to the cone generated by the columns of $B$ (i.e., $\stdCone_{\euc}^*$). That is,  
\begin{equation}\label{eq:theo1b}
S\hat \stdCone = \stdCone_{\euc}^*.
\end{equation}
From \eqref{eq:theo1a} and \eqref{eq:theo1b}, we conclude that $\stdCone'^*_{\euc} = \hat \stdCone$, which concludes the proof of item~\ref{theo:slack:ii}.
\end{proof}

\subsection{Self-dual cones and PSD slacks}
We now present a characterization of self-dual polyhedral cones: they are exactly the polyhedral cones that have a PSD slack. Furthermore, if at least one slack is PSD,  all of them can be made PSD by exchanging rows and rescaling columns if necessary.
\begin{theorem}\label{thm:self_psd}
Let $\stdCone \subseteq \Re^d$ be  a pointed full-dimensional polyhedral cone.
The following conditions are equivalent.
\begin{enumerate}[$(i)$]
	\item\label{thm:self_psd:1} There exists an inner product under which $\stdCone = \stdCone^*$.
	\item\label{thm:self_psd:2} $\stdCone$ is linearly isomorphic to a cone $\stdCone' \subseteq \Re^d$ satisfying 
	$\stdCone'^*_{\euc} = \stdCone'$.
	\item\label{thm:self_psd:3} Every slack matrix $M \in \S(\stdCone)$ is either PSD or there exists a permutation matrix $P$ and a positive diagonal matrix $D$ such that $PMD$ is PSD.
	\item\label{thm:self_psd:4} There exists $M \in \S(\stdCone)$ that is PSD.
\end{enumerate}
\end{theorem}
\begin{proof}
\fbox{\ref{thm:self_psd:1}$\Rightarrow$\ref{thm:self_psd:2}}
 By Proposition~\ref{prop:self_dual}, there exists a positive definite matrix $A$ such that $A\stdCone = \stdCone^*_{\euc}$.
 Since $A$ is positive definite, there exists a positive definite $B$ such that $B^2 = A$.
 Let $\stdCone' \coloneqq B \stdCone$.
 Then, from \eqref{eq:dual_iso} we obtain $(\stdCone')^*_{\euc} = B^{-1}(\stdCone^*_{\euc}) = B^{-1}(B^2\stdCone) = \stdCone'$.
 
\noindent\fbox{\ref{thm:self_psd:2}$\Rightarrow$\ref{thm:self_psd:1}} By assumption, there exists 
$B$ such that $B \stdCone = \stdCone'$ and $(\stdCone')^*_{\euc} = \stdCone'$.
Therefore, from \eqref{eq:dual_iso}, we obtain $B^{-\T}(\stdCone_{\euc}^*) =  (\stdCone')_{\euc}^* = \stdCone' = B\stdCone$. That is, 
$\stdCone_{\euc}^* = B^{\T}B\stdCone$, so by Proposition~\ref{prop:self_dual}, there exists an inner product under which $\stdCone = \stdCone^*$. 
 
\noindent\fbox{\ref{thm:self_psd:2}$\Rightarrow$\ref{thm:self_psd:3}} 
By Proposition~\ref{prop:slack_inv}, $\S(\stdCone) = \S(\stdCone')$ and we can consider that the slack matrices are constructed with, say, the Euclidean inner product.  Let $n$ be the number of extreme rays of $\stdCone'$. So, 
\begin{equation}\label{eq:mxy}
M = XY,
\end{equation}
where $X \in \Re^{n\times d}, Y \in \Re^{d\times n}$, the rows of $X$ generate the extreme rays of $\stdCone'$ and the columns of $Y$ generate the extreme rays of ${\stdCone'}^*_{\euc}$. 

Since ${\stdCone'}^*_{\euc}= \stdCone'$, we can permute the rows of $X$ such that the $i$-th row of $X$ generate the same extreme ray of $\stdCone'$ as the $i$-th column of $Y$. That is, there exists a permutation matrix $P$ and a diagonal matrix $D$ with positive entries, such that 
\begin{equation}\label{eq:pxdy}
(PX)^\T = YD.
\end{equation}
Then, \eqref{eq:mxy} and \eqref{eq:pxdy} imply that 
$PM = (PX)  (PX)^\T D^{-1}$.
Therefore, $PMD$, which is obtained from $M$ through exchanging rows and rescaling columns, is a symmetric positive semidefinite matrix.

\noindent\fbox{\ref{thm:self_psd:3}$\Rightarrow$\ref{thm:self_psd:4}} This follows because  exchanging rows and rescaling columns of any $M \in \S(\stdCone)$ still leads to a matrix belonging to $\S(\stdCone)$.

\noindent\fbox{\ref{thm:self_psd:4}$\Rightarrow$\ref{thm:self_psd:2}} Let $n$ be the number of extreme rays of $\stdCone$ and let $M \in \S(\stdCone)$ be PSD. Since $M$ is a slack matrix of a pointed $d$-dimensional cone, $M$ has rank $d$, see \cite[Lemma~13]{slack}. Therefore, there exists a rank $d$ matrix $X \in \Re^{n\times d}$ such that $M = XX^\T$.
Let $\stdCone' = \coner (X)$, i.e., the cone in $\Re^{d}$ generated by the rows of $X$.
By Theorem~\ref{theo:slack}, 
$\stdCone$ is linearly isomorphic to $\stdCone'$ and 
${\stdCone'}^*_{\euc} = \stdCone'$.
\end{proof}

\subsection{Slices of polyhedral cones and negatively self-polar polytopes}\label{sec:polytope}

Until now we have focused our attention on cones, with particular emphasis on polyhedral cones. However, our results have \textquotedblleft dehomogenized\textquotedblright \ versions that apply to convex bodies. In particular, we can adapt them to polytopes. The right version of self-duality we need in order to do this translation is \emph{negative self-polarity}.

Let $\ambSpace$ be a finite-dimensional Euclidean space equipped with some inner product $\inProd{\cdot}{\cdot}$.
Let $C \subseteq \ambSpace$ be a full-dimensional compact convex set such that $0 \in \reInt C$, where $\reInt C$ indicates the relative interior of $C$. Then, the polar $C^\circ$ is the set 
\begin{equation}\label{eq:polar}
C^\circ \coloneqq \{y \in \ambSpace \mid \inProd{x}{y} \leq 1, \forall x \in C \}
\end{equation}
and $C^\circ$ is also a  full-dimensional compact convex such that $0 \in \reInt C^\circ$ and $C^{\circ \circ} = C$, e.g., \cite[Corollary~14.5.1]{RT97}.
Analogously, we say that $C$ is \emph{negatively self-polar} if there exists an inner product under which $C = - C^\circ$.

Negatively self-polar polytopes are a very interesting class of polytopes, with deep connections to several areas of mathematics. Their study goes back to seminal works of Lov\' asz \cite{lovasz1983self} and Gr\" unbaum \cite{grunbaum1988selfduality} and are an integral part of a more general body of work on self-duality of polytopes. In \cite{MR4239237} one can find a modern survey on the area from this perspective, that revisits many of the results from \cite{BF76}.
In order to translate our previous results, one needs to adapt some of the definitions to this context.
Let $P\subset \Re^d$ be a full-dimensional polytope with extreme points $\{x_1,\ldots, x_n\}$. The facets of $P$ are cut out by linear inequalities $l_j(x)=b_j-\inProd{x}{y_j}$ for some $b_i \in \RR$ and
$y_j \in \RR^d$. Then, analogously to Definition~\ref{def:slack}, we can define a slack matrix of $P$ as the matrix $S(P)$ such that 
\[
M_{ij} = b_j - \inProd{x_i}{y_j}
\]
and define $\S(P)$ as the set of slack matrices of $P$. We note that  if we fix $M \in \S(P)$, every other matrix $\S(P)$ is obtained by permuting the rows or the columns of $M$, and scaling the columns by positive scalars. It can be seen that $\S(P)$ uniquely defines the polytope $P$ up to affine transformations (see \cite{MR4002715} for more on this topic). Note that if the origin is in the interior of $P$, we might choose the $b_j$ to be $1$ 
and $y_j$ to be the extreme points of the polar polytope $P^{\circ}$.

Next, suppose that $\ambSpace \subseteq \Re^{d+1}$ such that 
$\dim \ambSpace = d$ and let $w \in \ambSpace^\perp$ with $\inProd{w}{w} = 1$. Let $P \subseteq \ambSpace$ be a $d$-dimensional polytope with zero on the interior, with  extreme points $\{x_1,\ldots, x_n\}$. 
Let $\stdCone \subseteq \Re^{d+1}$ be the closed convex cone generated by $\{w\} + P$, that is,
\[
\stdCone \coloneqq \{\alpha w + \alpha x \mid x \in P, \alpha \geq 0 \}.
\]
 Then $\stdCone^*$ is the closed convex cone generated by $\{w\} -{P^\circ}$, where $P^\circ$ is computed with respect to $\mathcal{V}$ as in \eqref{eq:polar}\footnote{Here are a few details. Let $z \in \stdCone^*$ and $\pi$ denote the orthogonal projection onto $\{w\}^\perp$, so that $z = \pi(z) + \inProd{z}{w}w$ and $\inProd{z}{w} \geq 0$ (since $w \in \stdCone^*$).  From $z \in \stdCone^*$, it can be verified that $-\pi(z)/\inProd{z}{w} \in P^\circ$ if $\inProd{z}{w} > 0$. If $\inProd{z}{w} = 0$, then the assumption that $P$ is full-dimensional in $\mathcal{V}$ and $0 \in \reInt P$ implies that  $\pi(z) = 0$, so that $z = 0$. }.
 In particular, if $P$ is negatively self-polar, with respect to the inner product of $\Re^{d+1}$ restricted to $\ambSpace$, then $\stdCone$ is self-dual. Furthermore, since the extreme rays of $\stdCone$ and $\stdCone^*$  are generated by the  $w+x_i$ and $w-y_j$, respectively, where  $\{y_1,\ldots, y_m\}$ are the extreme points of $P^{\circ}$, we have 
 \begin{equation}\label{eq:sp_sk}
  \S(P) \subseteq \S(\stdCone).
 \end{equation}
Another important concept when dealing with polytopes is that of \emph{projective transformations}. Those are maps $\varphi: \RR^d \rightarrow \RR^d$ defined by
$$x \mapsto \frac{Ax+b}{a^\T x+\beta}$$ where
$A \in \RR^{d \times d}$, $a,b \in \RR^d$ and $\beta \in \RR$ and $$B\coloneqq \left(\begin{array}{cc} A & b \\ a^\T & \beta \end{array}\right)$$
is nonsingular. Another way of thinking of the map is simply the composition of the natural inclusion $\RR^d \rightarrow \{1\} \times \RR^d$, with the linear isomorphism $x \mapsto Bx$ and then with the perspective transformation $\RR^{d+1} \rightarrow \RR^d$ that takes $(x_0,x)$ to
$\frac{x}{x_0}$. When applied to a polytope $P$, the image $\varphi(P)$ is therefore obtained by applying $B$ to the cone $\stdCone$ generated by $\{1\} \times P$ and then slicing $B(\stdCone)$
with the plane defined by setting the first coordinate to $1$. We will say that two polytopes $P$ and $Q$ are \emph{projectively equivalent} if there exists a projective transformation $\varphi$ such that $\varphi(P)=Q$.

Under this setting, it turns out that 
if $P$ and $Q$ are projectively equivalent then $S(P)$ and $S(Q)$ are the same up to scaling rows by positive scalars, since by the above interpretation of $\varphi$ they are slack matrices of $\stdCone$ and $B(\stdCone)$ which have the same set of slack matrices by item~\ref{prop:slack_inv:2} of Proposition~\ref{prop:slack_inv}.

This gives us a tool to translate Theorem \ref{thm:self_psd} to the language of polytopes.

\begin{theorem}\label{thm:polytope_psd}
Let $P \subseteq \Re^d$ be a full-dimensional polytope with $0$ in its relative interior. 
Then, the following items are equivalent.
\begin{enumerate}[$(i)$]
	\item\label{thm:polytope_psd:1} $P$ is projectively equivalent to a polytope that is negatively self-polar with respect to some inner product.
    \item\label{thm:polytope_psd:4} $P$ is projectively equivalent to a polytope that is negatively self-polar with respect to the Euclidean inner product.
	\item\label{thm:polytope_psd:2} $P$ has a PSD slack.
\end{enumerate}
\end{theorem}
\begin{proof}

\fbox{\ref{thm:polytope_psd:4}$\Rightarrow$\ref{thm:polytope_psd:1}} This implication is straightforward.

\fbox{\ref{thm:polytope_psd:1}$\Rightarrow$\ref{thm:polytope_psd:2}}
If $Q$ is a negatively self-polar polytope, then letting $X$ be the matrix that has the extreme points of $Q$ in its rows and $A$ be the matrix of the underlying inner product, $M=[1\ X] \left(\begin{array}{cc} 1 & 0 \\ 0 &A \end{array}\right) [1\ X]^\T$ is a PSD slack for $Q$. 
If $Q$ is projectively equivalent to $P$, there is some positive diagonal matrix $D$ such that $DM \in \S(P)$. But then $DMD \in \S(P)$ is a PSD matrix.

\fbox{\ref{thm:polytope_psd:2}$\Rightarrow$\ref{thm:polytope_psd:4}} Suppose $M$ is a PSD slack for $P$. Then, $M$ is also a PSD slack for the cone $\stdCone$ in $\Re^{d+1}$ generated by $1\times P$, see \eqref{eq:sp_sk}.
By Theorem~\ref{thm:self_psd}, there exists a linear isomorphism $B$ sending $\stdCone$ to a cone $\stdCone'$ that is self-dual with respect to the Euclidean inner product. By composing $B$ with an orthogonal transformation we may assume that $\{1\} \times \{0\}$ is in the interior of $\stdCone'$, which means that the slice of $\stdCone'$ by the plane where the first coordinate is one is a compact polytope $Q$, and will be negatively self-polar with respect to the Euclidean product after projecting onto $\Re^d$ along the last $d$ coordinates. Moreover, as observed before, taking the cone over $P$, applying a linear transformation and then slicing with this plane is the same as making a projective transformation, so $Q$ is projectively equivalent to $P$.
\end{proof}

\section{Extreme rays of the DNN cone and completely positive semidefinite matrices}\label{sec:dnn_cs}
From Theorems~\ref{thm:self_psd} and \ref{thm:polytope_psd}, the slack matrices of self-dual polyhedral cones and negatively self-polar polytopes are either PSD or can be made 
PSD after exchanging rows and rescaling columns. They are also nonnegative so they are in fact doubly nonnegative matrices. 
In this section, we explore the connections between PSD slacks, doubly nonnegative and completely positive semidefinite matrices.

\subsection{Irreducible self-dual polyhedra and extreme rays of the DNN cone}
First, we will prove that PSD slacks of irreducible self-dual polyhedral cones are extreme rays of the doubly nonnegative cone.
We need a few preliminary results and discussions.
\paragraph{Irreducible cones and slack matrices}
We recall that a closed convex cone $\stdCone \subseteq \Re^d$ is said to be \emph{irreducible} (e.g., \cite{GT14}) if it is not possible to write $\stdCone$ as $\stdCone_1+\stdCone_2$, where $\stdCone_1$ and $\stdCone_2$ are nonzero closed convex cones satisfying $\spanVec(\stdCone_1) \cap \spanVec(\stdCone_2) = \{0\}$.
Otherwise, $\stdCone$ is said to be \emph{reducible} and we write $\stdCone = \stdCone_1 \oplus \stdCone_2$.
Irreducible cones are also called \emph{indecomposable}, see
 \cite{LS75,BF76}.
 
Irreducibility is a concept that does not depend on the underlying inner product. 
Accordingly, if we fix some inner product, the cones $\stdCone_1$ and $\stdCone_2$ are not necessarily orthogonal. However, the situation gets considerably simpler when the cones involved are polyhedral and self-dual, see also \cite[Section~3]{Ba78} for a related discussion.

\begin{lemma}\label{lem:self_orthogonal}
Suppose $\stdCone \subseteq \Re^d$ is a self-dual polyhedral cone under some inner product and that $\stdCone = \stdCone_1 \oplus \stdCone_2$. Then, 
$\stdCone_1 = \stdCone \cap \stdCone_2^\perp$. In particular, 
$\stdCone_1$ and $\stdCone_2$ are orthogonal under the same inner product.
\end{lemma}  
\begin{proof}
In view of \eqref{eq:lineality} and the self-duality of $\stdCone$, we have $\stdCone \cap -\stdCone = \stdCone^\perp$. So if $x \in \stdCone^\perp$, then $x \in \stdCone$ which forces $x$ to be zero. Therefore, $\stdCone \cap - \stdCone = \{0\}$ and $\stdCone$ must be pointed. With that, $\stdCone_1$ and $\stdCone_2$ are, in fact, faces of $\stdCone$.
Let $x \in \stdCone \cap \stdCone_2^\perp$ and $d_1 \coloneqq \dim\stdCone_1$ and $d_2 \coloneqq \dim \stdCone_2$, so that $d_1 + d_2 = d$ holds. Then, since 
$\stdCone = \stdCone_1 \oplus \stdCone_2$, we have $x = k_1 + k_2$, with $k_1 \in \stdCone_1$ and $k_2 \in \stdCone_2$. Since $x \in \stdCone_2^\perp$, we have $\inProd{x}{k_2} = 0$. Then, because 
$\stdCone$ is self-dual, $\inProd{k_1}{k_2} \geq 0$ holds, so  we have $k_2 = 0$.
Therefore, $\stdCone \cap \stdCone_2^\perp \subseteq \stdCone_1$.
Because $\stdCone$ is polyhedral, 
$\dim (\stdCone \cap \stdCone_2^\perp) = d - d_2 = d_1$ holds (e.g., \cite[Theorem~3]{Ta76}). Since 
$\stdCone_1$ and $\stdCone \cap \stdCone_2^\perp$ are both faces of $\stdCone$, this implies that $\stdCone_1 = \stdCone \cap \stdCone_2^\perp$.
\end{proof}

We recall that a matrix $A$ is said to be \emph{irreducible} if there is no permutation matrix $P$ such that $PAP^\T$ is in block upper triangular form, i.e., $PAP^\T =  \begin{psmallmatrix}
E & F \\ 
0 & G
\end{psmallmatrix}$.
If $A$ is symmetric, then $F = 0$, so for symmetric matrices irreducibility means that 
$PAP^\T$ is never block-diagonal for any permutation matrix $P$. 

Let $A \in \S^n$ and consider the support graph  $G(A)$ of $A$, which is constructed as follows. Let the vertex set be $V \coloneqq \{1,\ldots, n\}$ and $(i,j)$ belongs to the edge set if and only if $A_{ij} \neq 0$ and $i \neq j$. Then, the following fact is well-known.

\begin{lemma}[e.g., {\cite[Chapter~2]{BP94}, \cite[Section~1.4]{Va09}}]\label{lem:irreducibility}
	$A \in \S^n$ is \emph{irreducible} if and only if $G(A)$ is connected.
\end{lemma}
For self-dual  polyhedral cones, using  Lemma~\ref{lem:irreducibility} we can connect the notion of cone irreducibility with matrix irreducibility.
\begin{proposition}\label{prop:slack_irr}
	Let $\stdCone \subseteq \RR^d$ be a self-dual polyhedral cone and let $A \in \S(\stdCone)$ be a symmetric positive semidefinite slack matrix of $\stdCone$. Then, $\stdCone$ is irreducible if and only if $A$ is irreducible.
\end{proposition}
\begin{proof}
By item~\ref{thm:self_psd:2} of Theorem~\ref{thm:self_psd}, $\stdCone$ is linearly isomorphic to a convex cone $\stdCone'$ that is self-dual with respect the Euclidean inner product. By item~\ref{prop:slack_inv:2} of Proposition~\ref{prop:slack_inv}, $A \in \S(\stdCone')$.
Therefore, there exists a $n\times d$ matrix $X$ such that its rows generate the extreme rays of   $\stdCone'$, so that  $\rank(X) = d$  and
\begin{equation}\label{eq:axx}
A = XX^\T.
\end{equation}
Irreducibility is preserved under linear isomorphisms, so $\stdCone$ is irreducible if and only if $\stdCone'$ is irreducible.

With that in mind, first we prove that if $A$ is reducible then $\stdCone'$ is reducible. 
If $A$ is reducible, then 
there exists some permutation matrix $P$ such 
that $PAP^\T$ is block diagonal.
By \eqref{eq:axx}, we have  \[PAP^\T = PX (XP)^\T.\]
Therefore,  $PAP^\T$ being block diagonal means that the extreme rays of $\stdCone'$ can be re-arranged in two sets $\{u_1,\ldots, u_{\ell_1}\}\subseteq \RR^d$, $\{v_1,\ldots, v_{\ell_2}\}\subseteq \RR^d$ in such a way that $u_i$ and $v_j$ are orthogonal for all $i,j$.
Then, if $\stdCone_1$ is the cone generated by 
$\{u_1,\ldots, u_{\ell_1}\}$ and $\stdCone_2$ is the cone generated by  $\{v_1,\ldots, v_{\ell_2}\}\subseteq \RR^d$, we 
have $\stdCone' = \stdCone_1 \oplus \stdCone_2$.

Conversely, suppose that $\stdCone'$ is reducible.  Then, $\stdCone' = \stdCone_1 \oplus \stdCone_2$ with $\stdCone_1\cap \stdCone_2 = \{0\}$ and $\stdCone_1$, $\stdCone_2$  are nonzero cones with dimension $d_1$ and $d_2$, respectively. We also have $d_1 + d_2 = d$ and by Lemma~\ref{lem:self_orthogonal}, $\stdCone_1$ and $\stdCone_2$ are orthogonal with respect the Euclidean inner product.

Let $n_1$ and $n_2$ be the number of extreme rays of $\stdCone_1$ and $\stdCone_2$ and let $P$ be permutation matrix such that the first $n_1$ rows of $PX$ correspond to the extreme rays of $\stdCone_1$ and the last $n_2$ rows correspond to the extreme rays of $\stdCone_2$.
Then,
\[
PAP^\T = (PX)  (PX)^\T.
\]
Because of the orthogonality between $\stdCone_1$ and $\stdCone_2$, $PAP^\T$ is block diagonal, so $A$ is reducible.
\end{proof}

\paragraph{Extreme rays of the doubly nonnegative matrices}
Next, we review a characterization of the extreme rays of the doubly nonnegative matrices given in \cite{HL96}.

\begin{theorem}[{\cite[Theorem~2.1]{HL96}}]\label{theo:dnn_ext}
Suppose $A$ is a doubly nonnegative matrix of rank $k \geq 1$ and $A = XX^\T$ for some $n\times k$ matrix $X$.
Define 
\[
	\mathcal{W}_1 \coloneqq \{XQX^\T \mid Q \in \S^k \}, \qquad \mathcal{W}_2 \coloneqq \{Y \mid Y_{ij} = 0 \text{ if } A_{ij} = 0 \}.
\]
Then, $A$ generates an extreme ray of $\dnncone{n}$ if and only if 
$\mathcal{W}_1 \cap \mathcal{W}_2 = \{\lambda A \mid \lambda \in \Re\}$.
\end{theorem}
Theorem~\ref{theo:dnn_ext} will be one of the main tools we will use, so we need to establish a few results that make it easier to reason about the sets
$\mathcal{W}_1$  and $\mathcal{W}_2$ appearing therein.

In what follows, for  $v \in \RR^d$ we define
\[
\zeroes(v) \coloneqq \{i \mid v_i = 0 \}.
\]

\begin{lemma}\label{lem:slack}
Let $\stdCone \subseteq \RR^d$ be a pointed full-dimensional polyhedral cone and let $A \in \S(\stdCone)$.
 Denote the $i$-th row of $A$ by $a^i$. Suppose that $v\neq 0$ belongs to the span of the rows of $A$. If
	for some $i$ we have 
	\[
	\zeroes(a^i) \subseteq \zeroes(v)
	\]
	then $v$ is a nonzero multiple of $a^i$.
\end{lemma}
\begin{proof}
	We have
	\[
	v = \sum _{j=1}^n \alpha _j a^j.
	\]
	By virtue of $A$ being a slack matrix, we can write $A = UW$ where the rows of $U$ generate the extreme rays of $\stdCone$ and the columns of $W$ generate the extreme rays of $\stdCone^*_{\euc}$, where $\stdCone^*_{\euc}$  is the dual cone of $\stdCone$ with respect the usual Euclidean inner product.
	
	Denote the rows of $U$ by $u^j$ and the columns of $W$ by $w^j$. Let $z  \coloneqq \sum _{j=1}^n\alpha_j u^j$, Since 
	$A = UW$, we have $a^j = u^jW$ for every $j$. This implies that $v = zW$ and	
	 the $j$-th component of $v$
	satisfies
	\begin{equation}\label{eq:lem:slack}
	v_j = v{e^j} = zWe^j  =  z w^j 
	\end{equation}
	where $e^j$ is the $j$-th (column) unit vector. Just to avoid confusion, we remark that we are seeing $v$ and $z$ as row vectors, so $v{e^j}$ and $z w^j$ are indeed scalars.
	Then, the condition $\zeroes(a^i) \subseteq \zeroes(v)$ together with \eqref{eq:lem:slack} imply that  $z$ is orthogonal to every extreme ray of $\stdCone^*_{\euc}$ that is orthogonal to $u^i$ (and potentially more). 
	
	We recall that every extreme ray of $\stdCone$ is exposed 
	by a facet of $\stdCone^*_{\euc}$ and a facet of  $\stdCone^*_{\euc}$ is a $(d-1)$-dimensional polyhedral cone, e.g., \cite[Theorems~2 and 3]{Ta76}.
	So $z$ is also orthogonal to (at least) $(d-1)$ linearly independent extreme rays of 
	$\stdCone^*_{\euc}$ among the ones that are orthogonal to $u^i$.
	This means that $z$ is in the space spanned by $u^i$, that is, $z$ is a multiple of $u^i$.
	Since $v$ is nonzero, $z$ must be nonzero 
	as well, so, in fact, $z = \beta u^i$, where $\beta \neq 0$.
	Since $v = zW$, we obtain $v = \beta a^i$.
\end{proof}
\begin{lemma}\label{lem:dnn_w}
	Let $A$ be  a  $n\times n$ slack matrix such  
	that $A = XX^\T$, where $X$ is a $n\times d$ matrix  with $\rank(X) = d$.
	Define
	\[
	\mathcal{W}_1 \coloneqq \{XQX^\T \mid Q \in \S^d \}, \qquad \mathcal{W}_2 \coloneqq \{Y \mid Y_{ij} = 0 \text{ if } A_{ij} = 0 \}.
	\]
	If $B \in \mathcal{W}_1 \cap \mathcal{W}_2$, then the $i$-th row of $B$ is a multiple of the $i$-th row of $A$.  
\end{lemma}
\begin{proof}
	By the definition of $\mathcal{W}_1$, the space spanned by the columns of $B$ is contained in the space spanned by the columns of $X$. Now, the space spanned by the columns of $A$ is also contained in the space spanned by the columns of $X$, but since $\rank(X) = \rank(A) =d$, these two spaces are in fact equal.
	So, every column of $B$ is spanned by the columns of $A$. And since $A$ and $B$ are symmetric the same is true for the rows. 
	
	Finally, denote the $i$-th row of $A$ and $B$ by $a^i$ and $b^i$, respectively.
	Since $B \in \mathcal{W}_2$, we have 
	$\zeroes(a^i) \subseteq \zeroes(b^i)$, so 
	Lemma~\ref{lem:slack} tells us that $b^i$ is a multiple of $a^i$.
\end{proof}

\paragraph{Piecing everything together}
\begin{theorem}\label{theo:slack_ext}
	Let $A \in \S^n$ be a positive semidefinite slack matrix associated to a self-dual irreducible polyhedral cone $\stdCone \subseteq \RR^d$.
	Then, $A$ is an extreme ray of $\dnncone{n}$.
\end{theorem}
\begin{proof}
	By item~\ref{thm:self_psd:2} of Theorem~\ref{thm:self_psd}, $\stdCone$ is linearly isomorphic to a convex cone $\stdCone'$ that is self-dual with respect the Euclidean inner product. By item~\ref{prop:slack_inv:2} of Proposition~\ref{prop:slack_inv}, $A \in \S(\stdCone')$.
	Therefore, there exists a $n\times d$ matrix $X$ such that its rows generate the extreme rays of   $\stdCone'$, so that $A = XX^\T$ and $\rank(X) = d$.		
	Consider the sets $\mathcal{W}_1$ and $\mathcal{W}_2$  of Lemma~\ref{lem:dnn_w}.
	Every $B \in \mathcal{W}_1 \cap \mathcal{W}_2$ is such that the $i$-th row of $B$ is a multiple of the $i$-th row of $A$. That is, 
	$b^i = \lambda _i a^i$, where $a^i$ and $b^i$ are the $i$-th rows of $A$ and $B$, respectively. 
	Since $A$ and $B$ are symmetric, 
	\[
	B_{ji} = B_{ij} = \lambda_i A_{ij} = \lambda_{j} A_{ji} = \lambda_{j} A_{ij}.
	\]
	If $A_{ij} \neq 0$, we must have $\lambda _{i} = \lambda_j$.
	Now, consider the support graph $G(A)$. Since 
	$A_{ij} \neq 0$ implies $\lambda _{i} = \lambda_{j}$, we have that $\lambda _{k}$ must be constant for all indices that belong to the same connected component of $G(A)$. 
	However, since $\stdCone$ is irreducible and irreducibility is preserved under linear isomorphism, $\stdCone'$ is irreducible. 
	Then, $A$ is irreducible by Proposition~\ref{prop:slack_irr}, so $G(A)$ has a single connected component, by Lemma~\ref{lem:irreducibility}.
	We conclude that all the $\lambda _i$'s must be the same and $B = \lambda A$ for some $\lambda \in \RR$. 
	By Theorem~\ref{theo:dnn_ext}, this implies that $A$ is an extreme ray of $\dnncone{n}$.
\end{proof}
Next, we discuss Theorem~\ref{theo:slack_ext} in view of what is known about the extreme rays of $\dnncone{n}$.
The fact is that there are still several gaps on our knowledge of the extreme rays of the doubly nonnegative cone. And, indeed, it is mentioned in \cite{BDS15} that the  extreme rays of $\dnncone{n}$ are not yet completely understood. 

For $n \leq 4$, we have $\dnncone{n} = \cppcone{n}$ and the extreme rays of $\dnncone{n}$ are 
exactly the rank~$1$ matrices in $\dnncone{n}$.
For $n = 5$, \cite[Theorem~3.1]{HL96} implies that the possible ranks of extreme rays of $\dnncone{5}$ are either $1$ or $3$ and \cite[Theorem~3.2]{HL96} tells us that $A \in \dnncone{5} $ with rank $3$ is an extreme ray 
if and only if the graph $G(A)$ is a cycle of length $5$. Here is an example of one such matrix.
\begin{example}[Self-dual cone over a pentagon]\label{ex:pentagon}
We start with a self-dual cone  $\stdCone$ constructed over a regular pentagon, as described in \cite[pg.~152]{BF76}. Its extreme rays are generated by
\[
\left(\cos ({2\pi i}/{5}),\sin ({2\pi i}/{5}),\sqrt{-\cos(4\pi/5)}\right), \qquad i = 0, \ldots, 4.
\]
Letting $U$ be the matrix that has those vectors in its row, we obtain the following slack matrix
\[
M = UU^\T = \begin{pmatrix}
1+\cos \! \left(\frac{\pi}{5}\right) & \frac{\sqrt{5}}{2} & 0 & 0 & \frac{\sqrt{5}}{2} 
\\
\frac{\sqrt{5}}{2} & 1+\cos \! \left(\frac{\pi}{5}\right) & \frac{\sqrt{5}}{2} & 0 & 0 
\\
0 & \frac{\sqrt{5}}{2} & 1+\cos \! \left(\frac{\pi}{5}\right) & \frac{\sqrt{5}}{2} & 0 
\\
0 & 0 & \frac{\sqrt{5}}{2} & 1+\cos \! \left(\frac{\pi}{5}\right) & \frac{\sqrt{5}}{2} 
\\
\frac{\sqrt{5}}{2} & 0 & 0 & \frac{\sqrt{5}}{2} & 1+\cos \! \left(\frac{\pi}{5}\right) 
\end{pmatrix}.
\]	
$G(M)$ is a $5$-cycle and therefore connected, so $\stdCone$ is irreducible 
by Proposition~\ref{prop:slack_irr}.
By Theorem~\ref{theo:slack_ext}, $M$ is an extreme ray of $\dnncone{5}$.
	
\end{example}

 In fact, all the non completely positive extreme rays of $\dnncone{5}$ are  generated by slack matrices of self-dual cones over a pentagon. In order to make this claim precise, we state the following result, which follows from \cite[Theorem~24]{slack}\footnote{The only caveat is that the slack matrices considered in \cite{slack} may have some redundancies as described in Remark~\ref{rem:def_slack}. However, with some effort and considering that those redundancies are reflected in the zero pattern of the matrices and $N$ does not have such redundancies, one may show that Theorem~\ref{thm:comb_slack} indeed holds as stated considering the definition of slack matrices as in Definition~\ref{def:slack}.}. \begin{theorem}\label{thm:comb_slack}
	Suppose that $M \in \Re^{n\times n}$ is a rank $d$ nonnegative matrix and that there exists a slack matrix $N$ of a pointed full-dimensional polyhedral cone $\stdCone \subseteq \Re^d$ such that $M_{ij} \neq 0 \Leftrightarrow N_{ij} \neq 0$. Then, 
	$M$ is also the slack matrix of a pointed full-dimensional polyhedral cone in $\Re^d$.
\end{theorem}
We can now prove the following curious characterization of the extreme rays of $\dnncone{5}$.
\begin{theorem}[Characterizing the extreme rays of $\dnncone{5}$]\label{theo:dnn5}
Let $X \in \dnncone{5}$.
Then, $X$ generates an extreme ray of $\dnncone{5}$ if and only if $X$ is either rank~$1$ or a slack matrix of a irreducible self-dual cone over a pentagon.
\end{theorem}
\begin{proof}
If $X$ is rank~$1$ then $X$ is an extreme ray of $\dnncone{5}$ since it is even an extreme ray of $\psdcone{n}$.
If $X$ is the slack matrix of a irreducible self-dual cone over a pentagon, then, by Theorem~\ref{theo:slack_ext}, it must be an extreme ray of $\dnncone{5}$.

Conversely, suppose that $X$ is an extreme ray of $\dnncone{5}$.  By \cite[Theorem~3.1]{HL96}, the rank 
of $X$ is either $1$ or $3$. If it is $1$, we are done. So, suppose that the rank of $X$ is $3$. Then, by \cite[Theorem~3.2]{HL96}, the graph $G(X)$ is a cycle of length $5$. 
The slack matrix constructed in Example~\ref{ex:pentagon} for the self-dual cone over a pentagon is also a $5$-cycle. So, permuting the rows of $U$ in Example~\ref{ex:pentagon} if necessary, there exists a least one slack matrix $M$ such that $X_{ij} \neq 0 \Leftrightarrow M_{ij} \neq 0$ holds.
By Theorem~\ref{thm:comb_slack}, $X$ must be the slack matrix of a pointed full-dimensional polyhedral cone $\stdCone \subseteq \Re^3$. Since $X$ is PSD and $G(X)$ is connected, $\stdCone$ is a irreducible self-dual (with respect to some inner product) cone with $5$ extreme rays. That is, it is a cone generated by some pentagon.
\end{proof}


Next, we discuss Theorem~\ref{theo:slack_ext} for the cases $n =6$ and $n = 7$.
For $\dnncone{6}$, the possible ranks of extreme rays are  $1$ or $3$ \cite[Theorem~3.1]{HL96}. We note that a $6 \times 6$  rank~$3$ slack matrix must be the slack of a polyhedral cone in $\Re^3$ with $6$ extreme rays (which is a consequence of Theorem~\ref{theo:slack}). However, a self-dual polyhedral cone in $\Re^3$ must have a odd number of extreme rays \cite[Theorem~3]{BF76}. Therefore, Theorem~\ref{theo:slack_ext}  gives us no extreme rays in $\dnncone{6}$.
However, Theorem~\ref{theo:slack_ext}  does provides new extreme rays for $n \geq 7$.

\begin{example}[New families of extreme rays of $\dnncone{7}$]\label{ex:ext}
%
	
	We have a family of $7\times 7$ rank $4$ matrices that are the slack matrices of self-dual polyhedral cones. This is given by the slack matrix of the cone over a self-dual roofed triangular prism. A concrete example is the cone generated by the seven rays
	\begin{equation}\label{eq:cone7}\left( 1 , -\frac{1}{\sqrt{2}} , \pm \sqrt{\frac{3}{2}} , 0 \right),  \left( 1 , \sqrt{2} , 0 , 0 \right), \left( 1 , -\frac{1}{\sqrt{2}} , \pm \sqrt{\frac{3}{2}} , -1 \right),  \left( 1 , \sqrt{2} , 0 , -1 \right),\left( 1 , 0 , 0 , 1 \right)
	\end{equation}
	which is self-dual (\cite{tiwary2014self,MR4239237}).
	{This is obtained by applying the  second construction described in the proof of \cite[Theorem~4.7]{MR4239237} to the case $k = 1$ to generate a negatively self-polar polytope $P \subseteq \Re^3$.} Then, proceeding as in Section~\ref{sec:polytope}, the corresponding cone in $\Re^4$ generated by $1\times P$ is self-dual and has the extreme rays as in \eqref{eq:cone7}.
	 This corresponds to the cone over the roofed triangular prism shown in Figure \ref{fig:roofed_prism}.
	\begin{figure}
		\begin{center}
			\includegraphics[width=4cm]{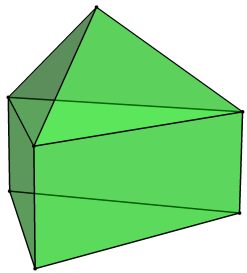}
		\end{center}
		\caption{Negatively self-polar roofed triangular prism} \label{fig:roofed_prism}
	\end{figure}
	The slack matrix will then provide an extreme ray for the $\dnncone{7}$. Explicitly, the  matrix
	$$S=\left(
	\begin{array}{ccccccc}
	3 & 0 & 0 & 3 & 0 & 0 & 1 \\
	0 & 3 & 0 & 0 & 3 & 0 & 1 \\
	0 & 0 & 3 & 0 & 0 & 3 & 1 \\
	3 & 0 & 0 & 4 & 1 & 1 & 0 \\
	0 & 3 & 0 & 1 & 4 & 1 & 0 \\
	0 & 0 & 3 & 1 & 1 & 4 & 0 \\
	1 & 1 & 1 & 0 & 0 & 0 & 2 \\
	\end{array}
	\right)$$
	is a rank $4$ extreme ray of $\dnncone{7}$. In fact, any matrix with the same rank and zero pattern (up to permutations) can be guaranteed to come from a self-dual cone combinatorially equivalent to this one, which follows from Theorem~\ref{thm:comb_slack}. 
	Since the rank of extreme rays in $\dnncone{7}$ is at most $5$, and we have a full characterization of the combinatorics of all self-dual polyhedral cones of dimension up to $5$ with at most seven vertices (see more details in Section~\ref{sec:exp}), 
	we know this is the only new family that our result guarantees to exist for $\dnncone{7}$.
\end{example}

Next we observe that the construction 
provided in \cite[Lemma 3.7]{HL96} does not necessarily correspond to self-dual polyhedra. So, indeed, the family of extreme rays coming from Theorem~\ref{theo:slack_ext} seems to be distinct. 

\begin{example}\label{ex:ext_not_slack}
In  \cite[Lemma 3.7]{HL96} the authors explicitly construct extreme matrices of the $\dnncone{n}$ of all allowable ranks. Applying their construction for $7 \times 7$ rank $4$ matrices we get the matrix
\[A=\left(
\begin{array}{ccccccc}
 2 & 1 & 0 & 0 & 2 & 0 & 2 \\
 1 & 2 & 1 & 0 & 0 & 0 & 0 \\
 0 & 1 & 2 & 2 & 0 & 2 & 0 \\
 0 & 0 & 2 & 3 & 1 & 3 & 1 \\
 2 & 0 & 0 & 1 & 3 & 1 & 3 \\
 0 & 0 & 2 & 3 & 1 & 4 & 0 \\
 2 & 0 & 0 & 1 & 3 & 0 & 4 \\
\end{array}
\right).\]
A slack matrix of a pointed full-dimensional cone $\stdCone \subseteq \Re^d$ must have rank $d$ and each row must have at least $d-1$ zeros, since each extreme ray of $\stdCone$ is orthogonal to a facet of $\stdCone^*$, e.g., \cite[Theorems~2 and 3]{Ta76}.
Therefore, $A$ is not a slack matrix of a polyhedral cone of dimension $4$, since there are two rows with only two zeroes. 

However, one can still interpret geometrically any doubly nonnegative matrix $A$. Since any such matrix of rank $k$ can be written as $A=XX^\T$, we can think of the $k$ dimensional cone generated by the rows of $X$, which is denoted by $\coner(X)$. The nonnegativity of $A$ is equivalent to $\coner(X) \subseteq \coner(X)^*$. When we have equality, we are in the scope of our previous result, and it seems plausible that any perturbation will break the inclusion, hence we have extremality. In Figure \ref{fig:polar_inclusion} we show a compact slice of $\coner(X)$. We can see that in the case of the extreme ray that does not come from a self-dual cone, the cone and its dual (represented here by their homogenisation) while not equal,  fit tightly with many of their facets spanning the same linear spaces.
\begin{figure}
\begin{center}
\includegraphics[width=8cm]{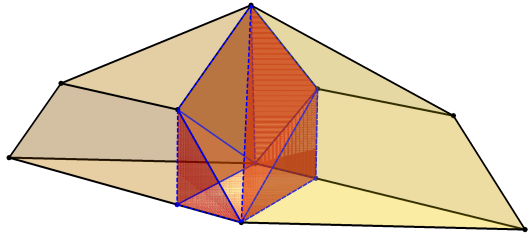}
\end{center}
\caption{A polytope $P$ in red included in its negative polar in yellow.} \label{fig:polar_inclusion}
\end{figure}
\end{example}
Example~\ref{ex:ext_not_slack} point us towards an interesting direction.
Although it is not possible to generate all extreme rays via self-dual cones,  we will conclude this subsection by showing that all DNN matrices arise from slacks of self-dual cones via  generalized congruence through a nonnegative matrix.
 This is possible because for any pointed full-dimensional cone $\stdCone$ satisfying $\stdCone \subseteq \stdCone^*_{\euc}$, one can ``squeeze'' a self-dual cone as follows, see also \cite[Theorem~5]{BF76} for a related result.
%
%
%
\begin{proposition}\label{prop:self_squeeze}
Let $\stdCone \subseteq \Re^d$ ($d \geq 2$) be a pointed full-dimensional polyhedral cone satisfying $\stdCone \subseteq \stdCone^*_{\euc}$. Then, there exists a self-dual polyhedral cone (with respect the Euclidean inner product) such that $\hat \stdCone$ such that $\stdCone \subseteq \hat \stdCone \subseteq \stdCone^*_{\euc}$.
\end{proposition}
\begin{proof}
This result essentially follows from {\cite[Corollary~7.3]{MR4239237}}, but we explain the details.
We recall that every pointed cone 
satisfies $\reInt \stdCone \cap \reInt \stdCone^* \neq \emptyset$ (e.g., \cite[Footnote~1]{LRS20}), so 
$\stdCone \subseteq \stdCone^*_{\euc}$ implies that, in fact, $\reInt \stdCone \subseteq \reInt \stdCone^*_{\euc}$ holds, see \cite[Corollary~6.5.2]{RT97}.

Let $w \in \reInt \stdCone$ of be of norm $1$ and let $U \in \Re^{d\times d}$ be any orthogonal matrix such that 
$Uw = (1,0,\ldots, 0)$ and 
let $\tilde \stdCone \coloneqq U \stdCone$.
In view of \eqref{eq:dual_iso}, $\tilde \stdCone^*_{\euc} = U^{-\T} \stdCone^*_{\euc} = U\stdCone^*_{\euc}$.
So $\tilde \stdCone$ is another cone such that $\tilde \stdCone \subseteq \tilde \stdCone^*_{\euc}$. 
Since $(1,0,\ldots, 0) \in \reInt \tilde \stdCone \subseteq \reInt \tilde\stdCone^*_{\euc}$ and 
$\tilde \stdCone$ is pointed, 
the slice
\[
C \coloneqq \{(1,x) \in \Re \times \Re^{d-1} \mid (1,x) \in \tilde\stdCone \}
\]
 is compact and generates $\tilde \stdCone$.
Then, projecting $C$ onto $\Re^{d-1}$ along the last $d-1$ coordinates we obtain a compact polyhedral set $P \subseteq \Re^{d-1}$, such that 
$1 \times P$ generates $\stdCone$. 
Since $(1,0,\cdots,0) \in \reInt \tilde \stdCone$ and $\tilde \stdCone$ is full-dimensional, $P$ has $0$ in its interior. Under these conditions, 
it can be verified that $1 \times (-P^\circ)$ generates $\stdCone^*_{\euc}$, e.g., see \cite[Proposition~3.3]{Vera14}.

Since $\tilde \stdCone \subseteq \tilde \stdCone^*_{\euc}$, $1\times P$ is contained in the convex cone generated by $1\times (-P^\circ)$. Therefore, 
$P \subseteq -P^\circ$.
By  {\cite[Corollary~7.3]{MR4239237}}, there exists a polyhedral set $Q$ such that $P \subseteq Q = -Q^\circ \subseteq -P^\circ$. $Q$ must be compact as well, since  $P^\circ$ is compact (recall that $P \subseteq \Re^d$ is full-dimensional in $\Re^d$ and has $0$ in its interior).
Let $\hat \stdCone$ be the convex cone generated by $1\times Q$. Then, since $Q$ is compact and has $0$ in its interior, $\hat \stdCone^*_{\euc}$ is generated by $1\times (- Q^\circ)$, so $\hat \stdCone$ is self-dual and satisfies
\[
\tilde \stdCone \subseteq \hat \stdCone = \hat \stdCone^*_{\euc} \subseteq \tilde \stdCone^*_{\euc}.
\]
Therefore, 
\[
\stdCone \subseteq U^{T} \hat \stdCone = (U^\T \hat \stdCone)^*_{\euc} \subseteq 
\stdCone^*_{\euc},
\]
so $U^{T} \hat \stdCone$ is a self-dual cone polyhedral sandwiched between $\stdCone$ and 
$\stdCone^*_{\euc}$.
%
%
%
%
%
%
%
\end{proof}

We can now prove the following. 
\begin{proposition} \label{prop:proj_extreme}
Let $A \in \dnncone{n}$ have rank $d$. Then, there exists a slack matrix $B \in \dnncone{m}$ of a self-dual polyhedral cone in $\Re^d$ and a nonnegative matrix $M \in \Re^{n\times m}$ such that $A = MBM^\T$. 
\end{proposition}
\begin{proof}
$A$ can be written as $A = XX^\T$, where $X \in \Re^{n\times d}$ and has rank $d$. 
Let $\stdCone = \coner(X)$, which is a cone in $\Re^d$. Since $A$  is nonnegative, we have $\stdCone \subseteq \stdCone^*_{\euc}$. Furthermore, $\stdCone$ is full-dimensional since $X$ has rank $d$. 
It must also be pointed, since the lineality space $\stdCone \cap -\stdCone$ is  equal to $\stdCone^{*\perp}_{\euc}$ (see \eqref{eq:lineality}) and $\stdCone^*_{\euc}$ contains $\stdCone$.

By Proposition~\ref{prop:self_squeeze}, there exists a self-dual polyhedral cone $\hat \stdCone$ such that $\stdCone \subseteq \hat \stdCone \subseteq \stdCone^*_{\euc}$.
Let $B \in \dnncone{m}$ be a PSD slack of $\hat \stdCone$, which exists by Theorem~\ref{thm:self_psd}. We may assume without loss of generality that the inner product under consideration is the usual Euclidean one so that $B = YY^\T$ for some $Y \in \Re^{m\times d}$ and the rows of $Y$ are the extreme rays of the cone $\hat \stdCone$.

Then, the condition $\stdCone \subseteq \hat \stdCone$ tells us that each row of $X$ is a nonnegative linear combination of rows of $Y$. That is, there exists a nonnegative matrix $M \in \Re^{n\times m}$ such that $MY = X$. Therefore
\[
A = MYY^\T M^\T = MBM^\T.
\]
\end{proof}

\begin{example}
Let us revisit Example \ref{ex:ext_not_slack} in light of Proposition~\ref{prop:proj_extreme}. Since the matrix $A$ of $\dnncone{7}$ presented there has rank $4$, according to Proposition \ref{prop:proj_extreme}, there must be some nonnegative matrix $M$ and a slack matrix $B$ such that $A=MBM^\T$.

That can be seen to be the case, for 
\[B=\left(
\begin{array}{cccccccc}
 2 & 1 & 0 & 0 & 2 & 2 & 0 & 2 \\
 1 & 2 & 1 & 0 & 0 & 1 & 0 & 0 \\
 0 & 1 & 2 & 2 & 0 & 0 & 2 & 0 \\
 0 & 0 & 2 & 3 & 1 & 0 & 3 & 1 \\
 2 & 0 & 0 & 1 & 12 & 8 & 4 & 0 \\
 2 & 1 & 0 & 0 & 8 & 6 & 2 & 0 \\
 0 & 0 & 2 & 3 & 4 & 2 & 4 & 0 \\
 2 & 0 & 0 & 1 & 0 & 0 & 0 & 4 \\
\end{array}
\right) \ \ \ M=
\left(\begin{array}{cccccccc}
 1 & 0 & 0 & 0 & 0 & 0 & 0 & 0 \\
 0 & 1 & 0 & 0 & 0 & 0 & 0 & 0 \\
 0 & 0 & 1 & 0 & 0 & 0 & 0 & 0 \\
 0 & 0 & 0 & 1 & 0 & 0 & 0 & 0 \\
 0 & 0 & 0 & 0 & \frac{1}{4} & 0 & 0 & \frac{3}{4} \\
 0 & 0 & 0 & 0 & 0 & 0 & 1 & 0 \\
 0 & 0 & 0 & 0 & 0 & 0 & 0 & 1 \\
\end{array}\right).\]

Graphically, this matrix $B$ is the slack matrix of a self-dual polyhedral cone with $8$ rays that lies sandwiched between  the cone of the rows of $X$ and its dual seen in Example \ref{ex:ext_not_slack}. Slicing the cones as done in that example we obtain the inclusions shown in Figure \ref{fig:polar_sandwich}, where the red section of the cone of the rows is included in the blue section of an autodual cone.

\begin{figure}
\begin{center}
\includegraphics[width=8cm]{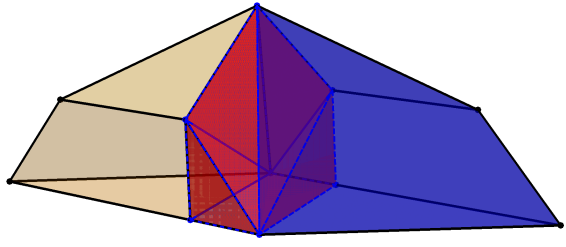}
\end{center}
\caption{A negatively selfpolar polytope $Q$ sandwiched between a polytope $P$ in red included in its negative polar in yellow.} \label{fig:polar_sandwich}
\end{figure}

The proof of Proposition \ref{prop:self_squeeze} in the paper \cite{MR4239237} is completely constructive, so we can actually construct these sandwiched cones explicitly, and therefore automatically construct the matrices $M$ and $B$. In practice, however, there seem to be simpler ways of finding ad hoc contructions for simple examples of extreme rays of $\dnncone{n}$ that are not slacks.

\end{example}

\subsection{Generating DNN matrices that are neither completely positive nor completly positive semidefinite}
PSD slack of self-dual polyhedral cones are doubly nonnegative. It is natural to wonder whether those matrices are, in fact, completely positive.
 Since $\cppcone{n} \subseteq \dnncone{n}$, a necessary condition for an extreme ray of  $\dnncone{n}$ to belong to $\cppcone{n}$ is that it must be an extreme ray of $\cppcone{n}$ as well. 
However, all the extreme rays of $\cppcone{n}$ have rank $1$.
On the other hand, a slack matrix of a pointed full-dimensional polyhedral cone  $\stdCone \subseteq \Re^d$ has rank $d$.

In view of these facts,
 Theorem~\ref{theo:dnn_ext} implies 
 that positive semidefinite slack matrices of irreducible self-dual polyhedral cones of dimension $d \geq 3$ are doubly nonnegative matrices that are not completely positive. 
 That is, if we are able to generate irreducible self-dual polyhedral cones, we can use them to construct families of non-completely positive matrices that are doubly nonnegative.

An interesting generalization of the completely positive cone $\cppcone{n}$ is the completely positive semidefinite cone $\cscone{n}$. This is the cone of all $n \times n$ matrices $M$ such that there exist positive semidefinite matrices $A_1,...,A_n$ such that $M_{i,j}=\langle A_i, A_j \rangle$ for all $i,j$, where $\inProd{\cdot}{\cdot}$ is the Frobenius inner product, which is given by $\inProd{X}{Y} = \tr(XY)$, for $X,Y \in \S^n$. Note that there is no prescription on the size of the $A_i$.

The basic properties of this cone can be found in \cite{laurent2015conic}, with other relevant work in \cite{gribling2017matrices, prakash2018completely, AKN20, frenkel2014vector}. This is a notoriously hard cone to handle that plays a role in some quantum games and quantum information literature. We have the following inclusions
\[
\cppcone{n} \subseteq \cscone{n} \subseteq \dnncone{n}.
\] 
Furthermore, for $n \geq 5$ the inclusions are strict \cite{laurent2015conic, frenkel2014vector}. 

The remainder of this subsection will be focused on proving that 
non-diagonal PSD slacks of polyhedral cones are also not completely positive semidefinite, thus significantly strengthening our previous observation that the PSD slacks of irreducible polyhedral cones are are not completely positive.
In order to do so, we will need a few preliminary results.

\paragraph{Completely positive semidefinite slacks and slices of the PSD cone}
Our analysis starts with the observation that if a PSD slack $S$ of a self-dual cone $\stdCone$ is completely positive semidefinite, then the matrices that appear in the decomposition of $S$ can be used to construct a slice of the PSD cone that is isomorphic to $\stdCone$.
In what follows we say that a (not-necessarily full-dimensional) cone $\stdCone$ is \emph{self-dual on its span} if there is some inner product under which 
$\stdCone = (\stdCone^*)\cap \spanVec\stdCone$.
 
\begin{lemma}\label{lem:psd_slice}
Let $S$ be a PSD slack matrix of a self-dual polyhedral cone $\stdCone \subseteq \Re^d$. Suppose that $S$ is completely positive semidefinite and let $A_1,\ldots, A_n \in \psdcone{k}$ be such that 
$S _{ij} = \inProd{A_i}{A_j}$ holds for all $i,j$.
Let $\bar \stdCone$ be the convex cone generated by the $A_i$, i.e., $\bar \stdCone = \{\sum _{i=1}^n \alpha _i A_i \mid \alpha_{i} \geq 0, \forall i =1,\ldots, n\}$.
 Let $H$ be the span of the $A_i$. 
Then, the following statements hold.
\begin{enumerate}[$(i)$]
	\item\label{lem:psd_slice:1}  $\stdCone$ and $\bar \stdCone$ are linearly isomorphic.
	\item\label{lem:psd_slice:2} $\bar \stdCone$ is self-dual on its span with respect the Frobenius inner product.
	\item\label{lem:psd_slice:3} $\bar \stdCone = \psdcone{k} \cap H$ holds.

\end{enumerate}
\end{lemma} 
\begin{proof}
In view of Theorem~\ref{thm:self_psd},  $\stdCone$ is isomorphic to a polyhedral cone that is self-dual with respect the usual Euclidean inner product. Isomorphic cones share the same slack matrices (Proposition~\ref{prop:slack_inv}), so for the purposes of this proof, we may assume without loss of generality that $\stdCone$ is self-dual with respect the Euclidean inner product and $S = VV^\T$, where each row of $V \in \Re^{n\times d}$ generates a distinct extreme ray of $\stdCone$. Denote the $i$-th row of $V$ by $v^{i}$. We define the map $\varphi: \Re^d \to H$ such that 
\begin{equation}\label{eq:iso}
\varphi\left(\sum _{i=1}^n \alpha _i v^i\right) \coloneqq \sum _{i=1}^n \alpha _i A_i.
\end{equation}
This map is well-defined. In fact, if $\sum _{i=1}^n \alpha _i v^i = \sum _{i=1}^n \beta _i v^i$, this means $(\alpha-\beta)V=0$. Multiplying by $V^T$ we obtain
$(\alpha-\beta)S=0$. Since $S _{ij} = \inProd{A_i}{A_j}$, this translates to
\[
\inProd{\sum _{i=1}^n (\alpha _i-\beta_i) A_i}{A_j}=0, \quad \textup{for all } j.
\]
Since the $A_j$ span $H$, this means $\varphi(\alpha V)-\varphi(\beta V)=\sum _{i=1}^n (\alpha _i-\beta_i) A_i=0$ so $\varphi$ is indeed well-defined. Next, we verify that $\varphi$ is injective and, hence, a bijection to its image $H$. Just note that $\varphi (\alpha V) = 0$ means $\inProd{\sum \alpha_i A_i}{A_j} = 0$  for all $j$, hence $\inProd{\sum \alpha_i v^i}{v^j} = 0$ and since the $v_j$ span $\Re^d$ we have $\sum \alpha_i v^i =0$.
We conclude that $\varphi(\stdCone) = \bar{\stdCone}$ and $\stdCone$ is linearly isomorphic to $\bar{\stdCone}$, which proves item~\ref{lem:psd_slice:1}.

We move on to item~\ref{lem:psd_slice:2}.
The span of $\bar \stdCone$ is $H$, so in order to show that 
$\bar \stdCone$ is self-dual on its span, we will check that 
$\bar \stdCone = \bar \stdCone^* \cap H$, where $\bar{\stdCone}^*$ is computed with respect the Frobenius inner product. 

The $A_i$ are PSD matrices so $\inProd{A_i}{A_j} \geq 0$ holds for all $i,j$. This leads to the inclusion $\bar \stdCone \subseteq \bar \stdCone^* \cap H$.
Conversely, let $B \in \bar \stdCone^* \cap H$.
In particular, $B$ is a linear combination of the $A_i$, i.e., $ B = \sum _{i=1}^{n} \alpha _i A_{i} $.
Now, let $b \coloneqq \sum _{i=1}^n \alpha _i v^{i}$. 
We recall that $S = VV^{\T}$ and $S_{ij} =  \inProd{A_i}{A_j} = \inProd{v^i}{v^j}$. Therefore, since $B \in \bar{\stdCone}^* \cap H$, we have $b \in \stdCone^*$.
Since $\stdCone$ is self-dual, we have $b \in \stdCone$ and there are nonnegative $\beta _{i}$'s such that 
\[
b = \sum _{i=1}^n \alpha _i v^{i} = \sum _{i=1}^n \beta _i v^{i}. 
\]
By the definition of $\varphi$ (see \eqref{eq:iso}), we have $\varphi(b) = B$ and $\varphi(v^i) = A_i$ for all $i$, so 
\[
B = \sum _{i=1}^n \alpha _i A^{i} = \sum _{i=1}^n \beta _i A^{i},
\]
which shows that $B \in \bar \stdCone$ and concludes item~\ref{lem:psd_slice:2}.

For item~\ref{lem:psd_slice:3}, we first observe that since the $A_i$ are PSD matrices, we have the inclusion  $\bar \stdCone \subseteq \psdcone{k} \cap H$.
Conversely, if $X \in  \psdcone{k} \cap H$, since $X$ and the $A_i$ are PSD we have $\inProd{X}{A_i} \geq 0$ for every $i$, which implies that $X \in \bar \stdCone^* \cap H$. By item~\ref{lem:psd_slice:2}, $\bar \stdCone^* \cap H = \bar \stdCone$, so $X \in \bar \stdCone$.
\end{proof}

Therefore, in order for a PSD slack matrix to be completely positive semidefinite, 
the underlying cone must be linearly isomorphic to a linear slice of a positive semidefinite cone, self-dual with respect to its span. Note that every polyhedral cone can be written as a linear slice of a positive semidefinite cone in many different ways, so the imposition of self-duality with respect to the matrix inner product is a key requirement. 
For more on the question of deciding whether a slice of the positive semidefinite cone is polyhedral see, for instance, \cite{bhardwaj2015deciding}.

Since $\psdcone{k}$ is self-dual, Lemma~\ref{lem:psd_slice} is related to  \cite[Theorem 4]{BF76}, that characterizes codimension one self-dual slices of a self-dual cone $\stdCone$ as those for which $\stdCone \cap H = \pi_H(\stdCone)$, where $\pi_H$ is the orthogonal projection onto $H$ with respect the underlying inner product. In fact codimension $1$ does not really matter, and we can get rid of it. 
For the sake of self-containment, below we  show the generalized version, but it also follows  from standard results such as \cite[Corollary 16.3.2]{RT97}.

\begin{theorem} \label{thm:selfdualslice}
	Let $\stdCone \subseteq \Re^d$ be a self-dual cone with respect  to some inner product $\inProd{\cdot}{\cdot}$ and let $H \subseteq \Re^d$ be a subspace. Then, the following statements are equivalent.
	\begin{enumerate}[$(i)$]
		\item\label{thm:selfdualslice:1} The span of $\stdCone \cap H$ is $H$ and $\stdCone \cap H$ is self-dual on its span with respect to $\inProd{\cdot}{\cdot}$.
		\item\label{thm:selfdualslice:2} $\pi_{H}(\stdCone) = \stdCone \cap H$, where the orthogonal projection is computed with respect to $\inProd{\cdot}{\cdot}$.
	\end{enumerate}
\end{theorem}
\begin{proof}
\fbox{\ref{thm:selfdualslice:1}$\Rightarrow$\ref{thm:selfdualslice:2}}	
   Since $\pi_{H} (\stdCone \cap H) = \stdCone \cap H$ we obtain the inclusion 
   $\stdCone \cap H \subseteq \pi_{H}(\stdCone)$.
   Conversely, suppose that $y$ is such that $y=\pi_H(x)$ for some $x \in \stdCone$. For every $z \in \stdCone \cap H$ we have $\langle z,y \rangle = \langle z,x \rangle$, and that is nonnegative by the self-duality of $\stdCone$. Therefore, $y$ belongs to the dual of $\stdCone \cap H$ and to $H$ (which is the span of $\stdCone \cap H$), hence belongs to $\stdCone \cap H$ since $\stdCone \cap H$ is self-dual on its span.

\fbox{\ref{thm:selfdualslice:2}$\Rightarrow$\ref{thm:selfdualslice:1}}	
First, we observe that
$\stdCone^\perp \subseteq \stdCone ^*$ always holds, so 
 the condition $\stdCone = \stdCone^*$ implies that $\spanVec \stdCone = \Re^d$.
Then, since  $\stdCone \cap H = \pi_H(\stdCone)$ and $\pi_{H}$ is a linear map, we have $\spanVec (\stdCone \cap H) = \pi _H(\spanVec \stdCone) = \pi _H(\Re^d) = H$.
	
	Next, the fact that $\stdCone \cap H$ is contained in its dual follows from the self-duality of $\stdCone$, so we only have to prove that $(\stdCone \cap H)^* \cap H \subseteq \stdCone \cap H$.
So, suppose that $z \in (\stdCone \cap H)^* \cap H $ and let $x \in \stdCone$. 
Then, since $\stdCone \cap H = \pi_H(\stdCone)$,
\[
0 \leq \inProd{z}{\pi_H(x)} = \inProd{\pi_H(z)}{x} = \inProd{z}{x},
\]
which shows that $z \in \stdCone^*$ and, therefore, $z \in \stdCone$ by the self-duality of $\stdCone$.  In conclusion,  $z \in \stdCone \cap H $ as intended.
\end{proof}

\paragraph{Projectional exposedness}
Next, we need a detour on the notion of \emph{projectional exposedness}. 
This notion has its origin in \cite{borwein1981regularizing} and was studied in \cite{barker1987projectionally, poole1988projectionally, sung1990study,LRS20}.  A cone $\stdCone$ is said to be
\emph{projectionally exposed} if for every face $F$ of $\stdCone$, there exists a linear projection $P_F$, i.e., a linear map (not necessarily self-adjoint) satisfying $P_F^2 = P_F$  such that $P_F\stdCone = F$.
If all those projections can be taken to be orthogonal, then the cone is said to be \emph{orthogonally projectionally exposed} or o.p.-exposed.
Examples of such cones include all symmetric cones \cite[Proposition 33]{lourencco2021amenable} of which $\RR_+^k$ and $\psdcone{k}$ are particular examples.  

Projectionally exposed cones are not only facially exposed \cite[Corollary~4.4]{sung1990study} but also satisfy a condition called \emph{amenability} \cite[Proposition~9]{lourencco2021amenable}. Amenability and projectional exposedness coincide for cones of dimension four or less \cite[Corollary~6.4]{LRS20} and, as of this writing, it is an open problem to exhibit an amenable cone that is not projectionally exposed, see \cite{LRS20}.

In any case, it turns out that projectional exposedness  interacts nicely with self-duality. A caveat is that while the notion of projection does not depend on the underlying inner product, that is not the case for orthogonal projections, i.e., a projection may become orthogonal under a suitable inner product. 
To avoid confusion, we clarify that in the definition of o.p-exposedness, the inner product is fixed for all the faces. That is, in order for a cone $\stdCone$ to be o.p-exposed, there must be an inner product $\inProd{\cdot}{\cdot}$ depending on $\stdCone$ only, under which  all faces $F$ of $\stdCone$ satisfy $\pi(\stdCone) = F$, where $\pi$ depends on $F$ and is an orthogonal projection computed with respect to $\inProd{\cdot}{\cdot}$.
With this in mind, in what follows we will be explicit and emphasize the choice of inner product. 
\begin{proposition}\label{prop:opexposed}
	Let $\stdCone \subseteq \Re^d$ be a closed convex cone which is  o.p.-exposed and self-dual with respect the same inner product $\inProd{\cdot}{\cdot}$. If $H \subseteq \Re^d$ is a subspace  such that the span of $\stdCone \cap H$ is $H$ and $\stdCone \cap H$ is self-dual on its span under $\inProd{\cdot}{\cdot}$  then $\stdCone \cap H$ is  o.p.-exposed with respect to $\inProd{\cdot}{\cdot}$.
\end{proposition}
\begin{proof}
The faces of $\stdCone \cap H$  are of the form $F \cap H$, where $F$ is a face of $\stdCone$, e.g., \cite[Section~IV]{Du62}.
So, let $F$ be an arbitrary face of $\stdCone$ and let $L$ be its span.
We will show that the orthogonal projection onto $H \cap L$ sends $\stdCone \cap H$ to $F \cap H$, where the orthogonal projection is computed with respect to $\inProd{\cdot}{\cdot}$, which is the same inner product that turns $\stdCone$ into a self-dual cone.
	
	Since $\stdCone$ is o.p.-exposed, we have 
	\begin{equation}\label{eq:proj_l}
	\pi_{L}(\stdCone)=F,
	\end{equation} where $\pi_L$ is the orthogonal projection onto $L=\spanVec F$. 
	We also know by Theorem \ref{thm:selfdualslice} that 
	\begin{equation}\label{eq:proj_h}
	\pi_H(\stdCone)=\stdCone \cap H,
	\end{equation}
	where $\pi_{H}$ is the orthogonal projection onto $H$.
	Starting with any $x_0 \in \stdCone \cap H$ define
	$y_{i+1}=\pi_L(x_i)$, $x_i=\pi_H(y_i)$, for $i\geq 1$. By \eqref{eq:proj_l} and \eqref{eq:proj_h}, all the $x_i$ and $y_i$ remain in $\stdCone$. Moreover, von Neumann's Theorem for alternate projections onto closed subspaces \cite[Theorem 13.7]{vonNeumann} implies that both the sequences $x_n$ and $y_n$ converge to $\pi_{H\cap L}(x_0)$ which, since $\stdCone$ is closed, implies $\pi_{H \cap L}(x_0)$ is in $\stdCone$, hence in $\stdCone \cap H \cap L$. However, 
	since $F$ is a face of $\stdCone$ and $L$ is its span, we have  $\stdCone \cap L = F$, so we conclude that $\pi_{H \cap L}(x_0) \in F \cap H$.
	Therefore, $\pi_{H \cap L}(\stdCone\cap H) \subseteq F \cap H$. Since $\pi_{H\cap L}(F \cap H) = F\cap H$, we obtain $\pi_{H\cap L}(\stdCone\cap H) = F \cap H$, as intended.
\end{proof}

\paragraph{The main result}
From Proposition~\ref{prop:opexposed}, a self-dual slice of a o.p.-exposed self-dual cone must also be o.p.-exposed. 
But pointed o.p.-exposed 
polyhedral cones were shown in 
\cite[Theorem~3.7]{sung1990study} to be simplicial, i.e., linearly isomorphic to the nonnegative orthant, see also \cite{barker1987projectionally}\footnote{This result was first announced in \cite{barker1987projectionally}, but the proof seems to have a gap as described in \cite[pg.~233]{sung1990study}.}.
Piecing everything together, we obtain  the following result.

\begin{theorem}\label{theo:diag_slack}
	Let $S$ be a PSD slack matrix of a self-dual polyhedral cone $\stdCone$. If $S$ is not diagonal, then it is not completely positive semidefinite.
\end{theorem}
\begin{proof}
	We prove the contrapositive statement, so suppose that $S$ is completely positive semidefinite.
	By Lemma \ref{lem:psd_slice}, $\stdCone$ is linearly isomorphic to a slice of $\psdcone{k}$ which must be self-dual with respect the Frobenius inner product.
	$\psdcone{k}$ is also o.p-exposed with respect the Frobenius inner product, which can be inferred from \cite[Example~3.1]{borwein1981regularizing} or from \cite[Proposition~33 and Equation~(18)]{lourencco2021amenable}\footnote{Equation~(18) indicates the inner product used in Proposition~33 of \cite{lourencco2021amenable}. For the algebra of real symmetric matrices, this inner product becomes $\inProd{X}{Y} = \tr(X\circ Y)$, where $X\circ Y \coloneqq (XY +YX)/2$ is the Jordan product in $\S^k$. With that, indeed $\inProd{X}{Y} = \tr(XY)$ holds.}.
	But then, by Proposition \ref{prop:opexposed}, this slice must be o.p.-exposed,
	which would imply 
	 that is simplicial, i.e.,  isomorphic to a nonnegative orthant, by 
	 \cite[Theorem~3.7]{sung1990study}.
	 Therefore, $S$ must be diagonal.	
\end{proof}
In view of Theorem~\ref{theo:slack}, a pointed full-dimensional polyhedral cone $\stdCone \subseteq \Re^d$ has a diagonal slack if and only if $\stdCone$ is simplicial, i.e., isomorphic to a nonnegative orthant. In particular, irreducible cones in $\Re^d$ for $d \geq 3$ never have diagonal slacks. We register this as a corollary of Theorem~\ref{theo:diag_slack}.

\begin{corollary}\label{col:cps}
Let $\stdCone \subseteq \Re^d$ be a self-dual polyhedral cone under some inner product, with $d \geq 3$. If $\stdCone$ is not simplicial (or, in particular, if $\stdCone$ is irreducible), then none of the PSD slacks  of $\stdCone$ are completely positive semidefinite. 
\end{corollary}

%

\section{Numerical experiments and conclusions}\label{sec:exp}

The characterization given by Theorem~\ref{thm:self_psd} allows us to try to enumerate self-dual polyhedral cones in some circumstances. Note that a simple combinatorial consequence of Theorem~\ref{thm:self_psd} is that the support of some slack matrix of a self-dual polyhedral cone $\mathcal{K}$  must be symmetric and contain the diagonal. This corresponds to the existence of a bijection $\varphi$ from the extreme rays of $\mathcal{K}$ to those of $\mathcal{K}^*$ (which we identify with facets of $\mathcal{K}$)  such that $v \subseteq \varphi(w)$ if and only if $w \subseteq \varphi(v)$ and such that $v \not \in \varphi(v)$ for all $v,w$ extreme rays of $\mathcal{K}$. This is a purely combinatorial property, and we say that in this case $\mathcal{K}$ is \emph{strongly involutive combinatorially self-dual}. 

This is a very special type of combinatorial self-duality. We say that a polyhedral cone $\stdCone$ is \emph{combinatorially self-dual} if there exist a bijection from the set of its faces to itself that inverts inclusion. Equivalently,  $\stdCone$  is \emph{combinatorially self-dual} if and only if it is combinatorially equivalent to its dual, i.e. the face lattices of  $\stdCone$ and  $\stdCone^*$ are isomorphic. We recall that the face lattice of a polyhedral set is simply the set of its faces together with the partial order given by the inclusion. These notions all have analogues for polytopes, replacing duals by polars, and we will make use of polytopes and cones interchangeably throughout this section.

Combinatorial self-duality in general does not need to be even involutive \cite{grunbaum1988selfduality,jendrol1989non}, i.e., a combinatorially self-dual polyhedral cone may not have a slack matrix with symmetric support, much less be strongly involutive. 

The study of combinatorial self-duality in the context of polytopes has a long history, specially in three dimensions. Efforts to enumerate these go back to the 19th century \cite{kirkman1857xi} and a characterization of combinatorially self-dual $3$-polytopes can be found in \cite{archdeacon1992construction}. The strongly involutive version has been recently studied in \cite{bracho2020strongly}. Since strongly involutive combinatorially  self-dual polytopes are in one-to-one correspondence to strongly involutive combinatorially  self-dual polyhedral cones, one can use this body of literature to generate polyhedral cones that are candidates to being self-dual.

In general, we can make use of enumerations that have been made of combinatorial types of polytopes for fixed dimension and number of vertices. In \cite{brinkmann2013house} we can find complete enumerations of all self-dual $3$-polytopes up to $15$ vertices (up to combinatorial equivalence), while in \cite{FirschingComplete4polys2020} we can find the complete enumeration of all $4$-polytopes of up to $9$ vertices. There seems to be little literature regarding self-duality for polytopes of dimension higher than three so no efforts have been made to enumerate specifically those polytopes.
Using these lists as starting points we checked for strong involutive self-duality by simply using brute force and check for each case if there exists a permutation of the columns of the support of the slack matrix that makes it symmetric and nonzero in the diagonal, thus deriving a list of strongly involutive combinatorially self-dual polytopes.

An alternative approach in dimension $3$ is to use the results of 
\cite{bracho2020strongly} to generate directly only $3$-polytopes that are  strongly involutive combinatorially self-dual, as in that paper they provide a generating procedure guaranteed to generate all of them. We used this approach to double check the results obtained with the previous method and extend it to $3$-polytopes of up to $17$ vertices. 

For three dimensions the results can be seen in Table \ref{tab:sisd3} while for four dimensions we can find it in Table \ref{tab:sisd4}. These correspond to polyhedral cones of dimensions four and five respectively. Note that the case for three dimensional polyhedral cones was settled by \cite{BF76}: they have self-dual realizations if and only if they have an odd number of rays, which happens if and only if they are strongly involutive combinatorially self-dual (see also \cite[Corollary~4.2]{MR4239237}). 
{Here, by a \emph{self-dual realization} of a cone $\stdCone$ we mean a self-dual cone $\stdCone'$ that is combinatorially equivalent to $\stdCone$.}

\begin{table}
\begin{center}
\begin{tabular}{ccccccccccccccc}
vertices   &  $4$  &  $5$  & $6$  & $7$   &  $8$   &  $9$  & $10$   & $11$   & $12$ & $13$   & $14$   & $15$ & $16$ & $17$ \\
s.d.       &$1$  &  $1$  & $2$  & $6$   &  $16$   &  $50$  & $165$   & $554$   & $1908$ & $6667$   & $23556$   & $84048$ & ? & ? \\
 s.i.c.s.d   &$1$  &  $0$  & $1$  & $1$   &  $2$   &  $4$  & $11$   & $24$   & $72$ & $212$   & $674$   & $2195$ & $7447$ & $25529$  \\
\end{tabular}
\end{center}
\caption{Number of $3$-polytopes with between $4$ and $17$ vertices that are combinatorially self-dual and strongly involutive combinatorially self-dual.} \label{tab:sisd3}
\end{table}

\begin{table}
\begin{center}
\begin{tabular}{cccccc}
vertices                                     &   $5$  & $6$  & $7$   &  $8$   &  $9$   \\
s.d.                     & 1  &  1   & 3   & 13    &  69    \\
 s.i.c.s.d  &$1$  &  $0$  & $1$  & $1$   &  $4$    \\
\end{tabular}
\end{center}
\caption{Number of $4$-polytopes with between $5$ and $9$ vertices that are combinatorially self-dual and strongly involutive combinatorially  self-dual.} \label{tab:sisd4}
\end{table}

Having generated these candidates, one still has to verify if they are effectively self-dual, and find an actual self-dual realization. 
While being strongly involutive combinatorially self-dual is a necessary condition, it is a purely combinatorial one, and there is no reason to believe it to be sufficient.

One very naive approach to finding self-duality in our sense given a strongly involutive combinatorially  self-dual cone is to rely on Theorem \ref{thm:self_psd}. If we have a strongly involutive combinatorial self-duality in a $d$-polyhedral cone we have a slack matrix with symmetric support with nonzero elements in the diagonal. We want a nonnegative PSD matrix with rank $d$ and the same support. Searching for positive semidefinite matrices with fixed zero entries is achievable by semidefinite programming. The only problem is the rank condition, which is extremely hard to enforce. Therefore we chose not to enforce it and try to satisfy only the other conditions. \medskip

{\bf Basic semidefinite approach:}
Given a symmetric $0/1$ matrix $M$ with ones in the diagonal, corresponding to the support of the slack matrix guaranteed by the strongly involutive combinatorial self-duality we will solve the semidefinite program:
\begin{align*}
\max_X         \ \    & \sum_{i,j} X_{ij}  \\
\textrm{s.t. }  \    & X_{ij}=0 \textrm{ if } M_{ij}=0; \\
					& X_{ii}=1 \textrm{ for all } i; \\
                   & X \succeq 0.
\end{align*}
Note that not only are we not enforcing the rank, but we are also not directly enforcing the positivity of the non-zero entries, only maximizing their sum. We are also fixing the entries in the diagonal to be one to make the problem compact. Somewhat surprisingly this turns out to be sufficient in many instances.

Straightforwardly solving the program above with MOSEK 8.0.0.60 \cite{mosek} we already get in most cases of Table \ref{tab:sisd3} (all but around three hundred cases) matrices that are numerically convincingly rank $4$, and that have positive entries in all the non-zero prescribed positions. Moreover, whenever the solution from the above semidefinite program fails to be rank $4$ or nonnegative, one can try to change the objective to  $ \sum_{i,j} c_{ij} X_{ij}$ where the $c_{ij}$ are random positive numbers.

With that small change we get semidefinite slack matrices in all the instances of Table \ref{tab:sisd3}. In fact we can get in all cases matrices whose non-zero prescribed entries are at least $10^{-4}$, the first four eigenvalues roughly in the interval between $1$ and $9$ and all the rest having absolute value smaller than $10^{-13}$. Note that this accuracy can in practice essentially be arbitrarily improved by doing alternate projections between the set of matrices of rank at most $4$ and those with the right support.  

Although this does not immediately rigorously prove that all the strongly involutive combinatorially  self-dual polyhedral cones with up to $17$ vertices in dimension $4$ actually have self-dual realizations, since we are not working with exact certificates, it strongly hints at that result. In fact one immediately suspects the following to be true.

\begin{conjecture}
All $4$-dimensional strongly involutive combinatorially self-dual polyhedral cones have self-dual realizations.
\end{conjecture}

For higher dimensions the data is too small to formulate any educated guess. The $7$ examples in Table~\ref{tab:sisd4} all seem to have self-dual realizations, since the semidefinite approach yields similarly positive results, but in fact all but two of them are simply pyramids over examples already in Table~\ref{tab:sisd3} so we have only two new examples. One interesting avenue of progress would be to see if restricting to strongly involutive combinatorially self-dual polytopes or even simply combinatorially self-dual polytopes could make it possible to extend the enumeration efforts of polytopes to higher dimensions and higher number of vertices, as that would provide us with new novel extreme rays of $\dnncone{n}$ as well as possibly give us some evidence on which to base a more general version of the previous conjecture.

To end the section, one should point out that, although only numerically, the slack matrices attained allow one to derive an approximately negatively self-polar polytope (or self-dual polyhedral cone) for each of the instances. Below we offer an example of that. 
The code that implements the process above as well as the other matrices we obtained are  available in the following link.

\url{https://github.com/bflourenco/self_dual_slacks}

\begin{example}
Applying the algorithm described above to a certain $10 \times 10$ support matrix corresponding to a strongly involutive combinatorially  self-dual polytope, we obtained the approximate rank $4$ PSD matrix presented below.
\begin{small}

$$
M=\left(\begin{array}{cccccccccc}
1&0.85962&0.63085&0.60758&0&0&0.32899&0.63085&0&0.60758\\
0.85962&1&0.85962&0.41395&0.41395&0&0&0.85962&0&0.41395\\
0.63085&0.85962&1&0.60758&0.60758&0&0&0.63085&0.32899&0\\
0.60758&0.41395&0.60758&1&0&0&0.54149&0&0.54149&0\\
0&0.41395&0.60758&0&1&0.54149&0&0.60758&0.54149&0\\
0&0&0&0&0.54149&1&0.70679&0.32899&0.70679&0.54149\\
0.32899&0&0&0.54149&0&0.70679&1&0&0.70679&0.54149\\
0.63085&0.85962&0.63085&0&0.60758&0.32899&0&1&0&0.60758\\
0&0&0.32899&0.54149&0.54149&0.70679&0.70679&0&1&0\\
0.60758&0.41395&0&0&0&0.54149&0.54149&0.60758&0&1
\end{array}\right)$$
\end{small}

This might not be a slack matrix of a polytope in the sense we introduced them in this paper, as we have scaled the rows and columns. However we can correct that by performing a singular value decomposition, $M=USV^\T$. There are only $4$ numerically positive singular values, so we can write
$M \approx U' U'^\T$ where $U'$ is the $10 \times 4$ matrix obtained multiplying each of the first four columns of $U$ by the square root of the corresponding entry of the diagonal of $S$. Finally, by Perron-Frobenius the first column of $U'$ has constant sign, so if we divide it by it we get a $10 \times 4$ matrix $\overline{W}=[1,W]$, where $\overline{W} \overline{W}^\T$ is still approximately PSD and has the same support. The rows of $W$ are then approximately the vertices of a negatively self-polar polytope. In this example we get $W^\T$ to be
\begin{small}
$$
\left(\begin{array}{cccccccccc}
-0.44578&-0.62782&-0.44578&0.21387&0.21387&1.5928&1.5928&-0.44578&1.5928&0.21387\\
0.42199& 0&0.1889&1.4123&-0.97561&-0.90532&0.62537&-0.6109&0.27994&-0.43672\\
0.46176&0 &-0.59634&-0.31113&-1.0676&0.19943&0.68431&0.13458&-0.88375&1.3787
\end{array}\right)$$
\end{small}

By looking at the zeros in the support we can reconstruct and plot the facets and we obtain in this case the polytope $P$ in Figure \ref{fig:num_selfdual_poly}. The cone generated by $\{1\} \times P$ is then numerically self-dual.
\begin{figure}
\begin{center}
\includegraphics[width=6cm]{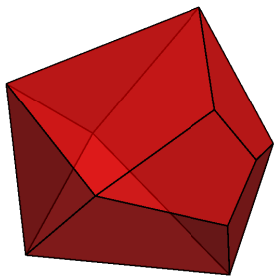}
\end{center}
\caption{A negatively self-polar polytope $P$ obtained numerically from our procedure.} \label{fig:num_selfdual_poly}
\end{figure}
\end{example}

\section*{Acknowledgements}
We thank the referees and the associate editor for their  comments, which helped to improve the paper.
\bibliographystyle{alpha}
\bibliography{selfdual}{}

\end{document}